\documentclass[11pt]{article}
\usepackage{amsmath}
\usepackage{amsthm}
\usepackage{amsfonts}
\DeclareMathOperator{\arccot}{arccot}
\usepackage{color}
\usepackage{hyperref} 
\usepackage{amsthm}

\usepackage{graphicx}

\usepackage [english]{babel}

\newtheorem{theo}{Theorem}[section]
\newtheorem{coro}[theo]{Corollary}
\newtheorem{lemma}{Lemma}[section]
\theoremstyle{definition}
\newtheorem{definition}{Definition}[section]
\theoremstyle{remark}
\newtheorem{remark}{Remark}[section]

\newcommand{\dsp}{\displaystyle}

\newcommand{\eProof}{\hfill {\tiny \mbox{$\Box$}}\medskip \smallskip} 
\renewcommand{\Re}{\mathop{\rm Re}}
\renewcommand{\Im}{\mathop{\rm Im}}
\newcommand{\const}{\mathop{\rm const}}

\newcommand{\C}{\mathbb{C}}
\newcommand{\D}{\mathbb{D}}
\newcommand{\R}{\mathbb{R}}
\newcommand{\T}{\mathbb{T}}

\newcommand{\fd}{\mathfrak{d}}
\DeclareMathOperator{\supp}{supp}

\begin{document}
\setlength{\parskip}{0.0cm}

\title{Extreme zeros in a sequence of para-orthogonal polynomials and bounds for the support of the measure}
\author{{A. Mart\'{i}nez-Finkelshtein$^{a}$, A. Sri Ranga$^{b}$ and D.O. Veronese$^{c}$} \\[1.5ex] 
{\small $^{a}$Departamento de Matem\'{a}ticas, Universidad de Almer\'{i}a, 04120 Almer\'{i}a,}\\
{\small  and Instituto Carlos I de F\'{\i}sica Te\'orica and Computacional, Granada University, Spain} \\[1ex]
{\small $^{b}$DMAp, IBILCE, } 
{\small UNESP - Universidade Estadual Paulista} \\
{\small 15054-000  S\~{a}o Jos\'{e} do Rio Preto,  SP,  Brazil} \\[1ex]
{\small $^{c}$ICTE, UFTM - Universidade Federal do Tri\^{a}ngulo Mineiro }\\
{\small 38064--200 Uberaba, MG, Brazil} \\}
%{\hspace*{3cm}}\\ }
\date{\today}
\maketitle 
%\tableofcontents
%\thispagestyle{empty}

%%%%%%%%%%%%%%%%%%%%%%%%%%%%%%%%%%%%%%%%%%%%%%%%%%%%%%%%%%%%%%%%

\begin{abstract} 
Given a non-trivial Borel measure $\mu$ on the unit circle $\mathbb T$, the corresponding reproducing (or Christoffel-Darboux) kernels with one of the variables fixed at $z=1$ constitute a family of so-called para-orthogonal polynomials, whose zeros belong to $\mathbb T$. With a proper normalization they satisfy a three-term recurrence relation determined by two sequence of real coefficients, $\{c_n\}$ and $\{d_n\}$, where $\{d_n\}$ is additionally a positive chain sequence. Coefficients $(c_n,d_n)$ provide a parametrization of a family of measures related to $\mu$ by addition of a mass point at $z=1$.

In this paper we estimate the location of the extreme zeros (those closest to $z=1$) of the para-orthogonal polynomials from the $(c_n,d_n)$-parametrization of the measure, and use this information to establish sufficient conditions for the existence of a gap in the support of $\mu$ at $z=1$. These results are easily reformulated in order to find gaps in the support of $\mu$ at any other $z\in \mathbb T$. 

We provide also some examples showing that the bounds are tight and illustrating their computational applications.

\end{abstract}

%%%%%%%%%%%%%%%%%%%%%%%%%%%%%%%%%%%%%%%%%%%%%%%%%%%%%%%%%%%%%%%%

%%%%%%%%%%%%%%%%%%%%%%%%%%%%%%%%%%%%%%%%%%%%%%%%%%%%%%%%%%%%%%%
\setcounter{equation}{0}
\section{Introduction} \label{Sec-Intro}

Orthogonal polynomials on the unit circle (in short, OPUC), also known as Szeg\H{o} polynomials, are one of the most beautiful objects in Classical Analysis, with deep connections and applications in many areas of Mathematics and Engineering. Their systematic study, started by Szeg\H{o} \cite{Szego-Book} and Geronimus \cite{Geronimus-AMSTransl-1977}, has experimented recently an important burst of activity, stimulated in part by their relation with the spectral theory \cite{Simon-Book-p1,Simon-2011}.

Given a nontrivial probability measure $\mu$ on the unit circle $\T: = \{\zeta=e^{i\theta}\!\!: \, 0 \leq \theta \leq 2\pi \}$ the associated orthonormal OPUC $\varphi_n(z)=\kappa_n z^n +\text{lower degree terms}$ are defined by $\deg \varphi_n=n$, $\kappa_n>0$ and 
\[
    \int_{\T} \varphi_{n}(z) \overline{\varphi_{m}(z)} d \mu(z) = \int_{0}^{2\pi} \varphi_{n}(e^{i\theta}) \overline{\varphi_{m}(e^{i\theta})} d \mu(e^{i\theta}) = \delta_{nm} ,
\]
where $\delta_{nm}$ stands for the Kronecker delta. Among their fundamental properties is that all their zeros belong to the open unit disk $\D:= \{ z\in \C: \, |z|<1\} $, and that they satisfy the Szeg\H{o} recurrence,
\begin{equation} \label{Szego-A-RR}
\Phi_{n}(z) =  z \Phi_{n-1}(z) - \overline{\alpha}_{n-1}\, \Phi_{n-1}^{\ast}(z), \quad 
n \geq 1,
\end{equation}
given here in terms of the monic OPUC $\Phi_{n}(z) = \varphi_{n}(z)/\kappa_n$, $n \geq 0$. The coefficients  $\alpha_{n-1} = - \overline{\Phi_{n}(0)}$ are the Verblunsky coefficients, and $\Phi_{n}^{\ast}(z) = z^{n} \overline{\Phi_{n}(1/\bar{z})}$.  It is well known that  $|\alpha_n| < 1$, $n \geq 0$, and that $\{\alpha_n \}_{n=0}^\infty \subset \D^\infty \leftrightarrow \{\mu \}$ is a bijection between the set of Verblunsky coefficients and positive Borel measures on $\T$ (see e.g.~\cite{Simon-Book-p1}, as well as \cite{ENZG1}). Establishing the correspondence $\rightarrow$ (resp., $\leftarrow$) is the \textit{direct} (resp., \textit{inverse}) spectral problem. In particular, recovering the support of the associated measure from the sequence of Verblunsky coefficients is an important part of the direct spectral problem.

In the classical situation of orthogonal polynomials on the real line, their zeros are strongly correlated with the support of the measure of orthogonality: they all belong to the convex hull of this support, and each polynomial has at most one zero in each gap of the support. This is not the case of the OPUC:  the behavior of their zeros in $\D$ can be very complicated, even for the simplest weights, see e.g.~\cite{MFMLS,MFMLS2}. Looking for sequences of polynomials related to $\mu$ and with zeros \emph{on} $\T$,  Jones, Nj\aa stad and Thron \cite{JoNjTh-1989} introduced the {\em para-orthogonal polynomials on the unit circle} (in short, POPUC): they are, up to normalization, polynomials of the form $\Phi_n(z) - \tau_{n}\Phi_n^{\ast}(z)$, where $|\tau_{n}| =1$. The definition $z\Phi_{n-1}(z) - \tilde{\tau}_{n-1}\Phi_{n-1}^{\ast}(z)$, where $|\tilde{\tau}_{n-1}| = 1$, is also used and is equivalent, since  
\[
     \Phi_n(z) - \tau_{n}\Phi_n^{\ast}(z) = (1+\tau_{n}\alpha_{n-1}) \Big[z\Phi_{n-1}(z) - \frac{\tau_{n}+\overline{\alpha}_{n-1}}{1+\tau_{n}\alpha_{n-1}}\Phi_{n-1}^{\ast}(z)\Big],
\]
and
\[
|\tau_{n}| =1 \quad \Leftrightarrow \quad \left|\frac{\tau_{n}+\overline{\alpha}_{n-1}}{1+\tau_{n}\alpha_{n-1}}\right| = 1.
\]
By \eqref{Szego-A-RR}, a POPUC of degree $n$ can be constructed from the set $\{\Phi_{j}\}_{j=0}^{n-1}$ of OPUC, or equivalently from the set of Verblunsky coefficients $\{\alpha_j\}_{j=0}^{n-1}$, by forcing the last coefficient $\alpha_{n-1}$ in \eqref{Szego-A-RR} to be on $\T$. For further information on POPUC see \cite{CantMoraVela-2002, DaNjVA-2003, Dimitar-Ranga-MN2013, Golinskii-2002, Simon-2011, Wong-2007} and the references therein. 

Zeros of all POPUC lie on the unit circle $\T$. In fact, Geronimus \cite{Geronimus-AM1946} (see also \cite{JoNjTh-BLMS1989}) gave a general construction of rational functions whose denominators are an arbitrary sequence of POPUC, and such that they converge to the Carath{\'e}odory function $\int_{\mathcal{C}}(z+\zeta)(z-\zeta)^{-1}d\mu(\zeta)$ in $\D$ (and by symmetry, also in the complement to $\overline{\D}$). For an \emph{arbitrary} sequence of POPUC not much else can be said.

In the case when $\mu$ is not supported on the whole $\T$, we would like to choose the sequence of POPUC whose zeros are consistent with the $\supp(\mu)$. Ideally, one needs to assure that the set of attracting points (strong limit points in the terminology of \cite{CantMoraVela-JAT2006}) of the zeros of the sequence  of para-orthogonal polynomials coincides with the support of the measure. This would allow one to characterize this support in terms of the asymptotic distribution of such zeros. This problem was essentially solved in the work of Cantero, Moral and Vel\'{a}zquez \cite{CantMoraVela-JAT2006}, where they put forward a  rather general construction of the sequence of POPUC with such a property, using the connection with the spectral theory of the (unitary) truncated CMV matrices.

As it was already observed by Golinskii \cite{Golinskii-2002}, a particular instance of the polynomials considered in \cite{CantMoraVela-JAT2006} is closely related to the  reproducing (also Christoffel-Darboux or CD) kernel associated to $\mu$. Recall that the CD kernel is
\begin{equation}\label{def:CDkernel}
K_{n}(w,z) := \sum_{j=0}^{n} \overline{\varphi_j(w)}\,\varphi_j(z) \,, \quad n \geq 0, 
\end{equation}
and the well-known Christoffel-Darboux formula (see, for example, \cite[Thm. 2.2.7]{Simon-Book-p1}) says that  for $z\neq w$,
\begin{equation}\label{CDformula}
  \begin{array}{ll}
    K_{n}(w,z)  & \dsp = \frac{ \overline{\varphi_{n+1}(w)}\,\varphi_{n+1}(z)- \overline{\varphi_{n+1}^{\ast}(w)}\,\varphi_{n+1}^{\ast}(z)} {\bar{w}\, z-1} \\[2ex]
    & \dsp = \frac{ \overline{w} z \overline{\varphi_{n}(w)}\,\varphi_{n}(z)- \overline{\varphi_{n}^{\ast}(w)}\,\varphi_{n}^{\ast}(z) } {\bar{w}\, z-1}\,,  \quad n \geq 0.
  \end{array}
\end{equation}
It shows that if we fix $w\in \T$, then up to a normalization factor, $(\bar{w}\, z-1) K_{n}(z,\,w)$ is a POPUC of degree $n+1$:
\begin{equation}
  \begin{array}{ll}
(\bar{w}\, z-1) K_{n}(z,\,w) &= \const \left( \Phi_{n+1}(z)- \tau_{n+1}(w)\,\Phi_{n+1}^{\ast}(z) \right)  \\[1ex]
&= \const \left(z \Phi_{n}(z)- w\tau_{n}(w)\,\Phi_{n}^{\ast}(z) \right), 
  \end{array}
\end{equation}
with
\begin{equation}\label{defTau}
\tau_{n}(w) = \overline{ \left(\frac{ \varphi_{n}^{\ast}(w)}{  \varphi_{n}(w)  }\right)} =   \frac{ \Phi_{n}(w)}{   \Phi_{n}^{\ast}(w)  }\in \T.
\end{equation}
These polynomials have an additional feature: they satisfy a bona fide three-term recurrence relation (see \cite[Thm.\,2.1]{Costa-Felix-Ranga-JAT2013}). A deeper analysis of sequences of para-orthogonal polynomials satisfying three term recurrence relations was carried out in \cite{BracRangaSwami-2015}.

In what follows, without loss of generality we will assume $w=1$, and write $\tau_{n} $ instead of $\tau_{n}(1) $. When the freedom in the selection of the parameter $w$ will be relevant, it will be explicitly discussed.

Paper  \cite{Costa-Felix-Ranga-JAT2013} also contained a convenient ``symmetrization'' of the polynomials $K_{n}(1,z)$; it was shown there that the appropriately normalized $K_{n}(1,z)$, that we denote by $R_n$ (see the precise definition in Section \ref{sec:CDandPositiveCS})  satisfy  a three term recurrence formula of the form 
\begin{equation}  \label{Eq-TTRR-Rn}
   R_{n+1}(z) = \left[(1+ic_{n+1})z + (1-ic_{n+1})\right] R_{n}(z) - 4d_{n+1} z R_{n-1}(z), \quad n\geq 0,
\end{equation}
with $R_{-1}(z) = 0$ and $R_{0}(z) = 1$. % and $R_{1}(z) = (1+ic_{1})z + (1-ic_{1})$, 
This can be considered as a generalization of the results found in Delsarte and Genin  \cite{DelsarteGenin-1986, DelsarteGenin-1988}.

Sequences   $\{c_n\}_{n=1}^{\infty}$ and $\{d_{n+1}\}_{n=1}^{\infty}$ are both real, and $\{d_{n+1}\}_{n=1}^{\infty}$ has an additional useful feature: it is a \emph{positive chain sequence}.  In other words, there exists a sequence $ \{g_n\}_{n\geq 1}$ with $0\leq g_1< 1$ and $0<g_n<1 $ for $n\geq 2$, called a \textit{parameter  sequence} for $\{d_{n+1}\}_{n=1}^{\infty}$, such that
\begin{equation}\label{def:CS}
d_{n+1} = (1-g_{n})g_{n+1}, \quad n \geq 1.
\end{equation}
We can have either a unique or an infinite number of parameter sequences corresponding to a given positive chain sequence. The value $g_1=0$ used in \eqref{def:CS} gives us the \textit{minimal parameter sequence} for $\{d_{n+1}\}_{n=1}^{\infty}$. The largest value $M_1$ that $g_1$ can assume such that the sequences $\{g_n\}_{n\geq 1}$ calculated from \eqref{def:CS} is still a parameter sequences, gives rise to the \textit{maximal parameter sequence} $\{M_n\}_{n\geq 1}$. Thus, the parameter sequence for $\{d_{n+1}\}_{n=1}^{\infty}$ is unique if and only if $M_1=0$; otherwise any value of $0\leq g_1\leq M_1$ generates a parameter  sequence, and we say that $\{d_{n+1}\}_{n=1}^{\infty}$ is a {\em non single parameter} (or non {SP}) positive chain sequence.
 
This notion makes sense also for finite sequences $\{d_{n+1}\}_{n=1}^{N}$, in which case $0\leq g_1< 1$,  $0<g_n<1 $ for $2 \leq n \leq N+1$ and we speak about \textit{finite positive chain sequences}, see \cite{IsmailLi-1992}.

The theory of chain sequences was introduced by Wall \cite{Wall-Book}, and the special version of {\em positive} chain sequences has been thoroughly explored by Chihara  and many others for studying the properties  of orthogonal polynomials defined on bounded intervals of the real line.  For many of the basic properties about  positive chain sequences we refer to Chihara \cite{Chihara-Book}.

Sequences  $\{c_n\}_{n=1}^{\infty}$ and $\{d_{n+1}\}_{n=1}^{\infty}$ can be easily calculated from the Verblunsky coefficients of the corresponding measure $\mu$ on $\T$ (and viceversa), providing an alternative parametrization of a family of measures related to $\mu$ by mass addition at $z=1$, see \cite{BracRangaSwami-2015,Castillo-Costa-Ranga-Veronese-JAT2014,Costa-Felix-Ranga-JAT2013} and Section \ref{sec:CDandPositiveCS} for further details.    

The zeros of $R_n$, generated by (\ref{Eq-TTRR-Rn}), have a strong resemblance with zeros of orthogonal polynomials on the real line: they belong to $\T$, are simple and interlace. In other words, if we denote the zeros of $R_n$ by $z_{n,j} = e^{i\theta_{n,j}}$, $j =1, 2, \ldots, n$, then (see e.g.~\cite{Dimitar-Ranga-MN2013}) we can number them in such a way that
\begin{equation} \label{Eq-InterlacingZeros-Rn}
    0 < \theta_{n+1,1} < \theta_{n,1} < \theta_{n+1,2} < \cdots < \theta_{n,n} < \theta_{n+1,n+1} < 2 \pi, \quad n \geq 1. 
\end{equation} 

This property and the three terms recurrence relation \eqref{Eq-TTRR-Rn} are an additional tool, missing for more general families of POPUC, that allows us to establish more precise results about the support of the orthogonality measure $\mu$ in terms of the zeros of $\{R_n\}_{n=0}^{\infty}$. In particular, we study the extremal zeros $z_{n,1}$ and $z_{n,n}$ and their asymptotic behavior, establishing bounds for the support of $\mu$ in terms of the sequences  $\{c_n\}_{n=1}^{\infty}$ and $\{d_{n+1}\}_{n=1}^{\infty}$ (and indirectly, in terms of the Verblunsky coefficients).  Considering para-orthogonal polynomials obtained from the rotated measures we also get information about  bounds for any gap within the support of the original measure.

\begin{definition}\label{def:scalingsequence}
	A positive sequence (finite or infinite) $\{q_{n+1}\}_{n=1}^N$ is called a \textit{scaling sequence} for the positive chain sequence $\{d_{n+1}\}_{n=1}^N$ if $q_{n+1} \leq 1$ for $1 \leq n \leq N$ and $\{ d_{n+1}/q_{n+1}\}_{n=1}^N$ is also a positive chain sequence.
\end{definition}
Obviously, $\{q_{n+1}\}_{n=1}^N$ with $q_{n+1}=1$ for $n=1,\dots, N$, is a (trivial) scaling sequence for any positive chain sequence. 
%, constructed by taking $\widehat{d}_{n+1}=d_{n+1}$.

For convenience, we introduce here some additional notation used throughout this manuscript. Open and closed arcs on the unit circle are denoted by
\begin{equation}\label{notationArcs}
\mathcal{A}(\vartheta_1, \vartheta_2)=\{ e^{i \theta}:\, \vartheta_1 < \theta < \vartheta_2   \} \quad \mbox{and} \quad  \mathcal{A}[\vartheta_1, \vartheta_2]=\{ e^{i \theta}:\, \vartheta_1 \le \theta \le \vartheta_2   \},
\end{equation}
for $0 < \vartheta_2 - \vartheta_1 \leq 2\pi$, respectively. Additionally, for $a, b\in \mathbb R$ and $q\in [0,1]$ we denote by $u^{(-)}(a,b,q)<u^{(+)}(a,b,q)$ the roots of the quadratic equation
\begin{equation}\label{mainpolynomial}
P(u):=(1-q) u^2 -(a+b)u + ab -q=0.
\end{equation}
It is straightforward to check that for $q\in [0,1)$ these roots are real and finite; for $q=1$ we assume $u^{(\pm)}(a,b,q)=\pm\infty$ when $\pm(a+b)\geq 0$.

The main results are presented in Sections \ref{Sec-BEZeros} and \ref{Sec-BSupportM}. One of them can be stated as follows:

\begin{theo} \label{Thm-Restriction-1-for-cn-dn2}
	Let $\mu$ be a nontrivial probability Borel measure on $\T$, being   $\{c_n\}$,  $\{d_{n+1}\}$  the coefficients of the corresponding recurrence relation \eqref{Eq-TTRR-Rn}. Given a scaling sequence $\{q_{n+1}\}$ for $\{d_{n+1}\}$, we define 
	\begin{equation} \label{alternativeU}
	u_{n+1}^{(\pm)}=u^{(\pm)}(c_n,c_{n+1},q_{n+1}), \quad n\geq 1. 
\end{equation}
	If  
	$$
0 \leq\vartheta_1 :=2 \arccot\left( \sup_{n\geq 1} u_{n+1}^{(+)}\right) < \vartheta_2 :=2 \arccot\left( \inf_{n\geq 1} u_{n+1}^{(-)}\right)\leq 2\pi,
	$$
then
\[   
   \supp (\mu- \mu(\{1\})\delta_1) \subseteq \mathcal{A}[\vartheta_1, \vartheta_2] .
\]
Here, $\delta_1$ is the Dirac delta at $z=1$.  
\end{theo}

\begin{remark}
It will become clearer later, the smaller the scaling sequence $q_{n+1}$ is, the smaller the size of the arc $\mathcal{A}[\vartheta_1, \vartheta_2]$; the trivial choice $q_{n+1}\equiv 1$ gives the worst estimate for $\supp (\mu)$.  
\end{remark}

In the next section it will be shown that the coefficients of \eqref{Eq-TTRR-Rn} can be obtained directly from the Verblunsky coefficients $\{\alpha_n\}$ of $\mu$. Thus, the theorem above gives indirectly a bound for $\supp \mu$ in terms of the  $ \alpha_n $'s. It can be seen also as a statement about the gap in $\supp(\mu)$ around $z=1$. More generally, we can locate a gap using a construction starting from the  $\{\alpha_n\}_{n=0}^{\infty}$ of a  nontrivial positive measure  $\mu$ on $\T$. Namely, for a parameter $0 <  \vartheta  \leq 2 \pi$, let the sequence $\{\tau_{n}^{(\vartheta )}\}_{n=0}^{\infty}$ be given recursively by 
\[
\tau_{0}^{(\vartheta )} = e^{i\vartheta } \quad \mbox{and} \quad 
\tau_{n}^{(\vartheta )} = e^{i\vartheta }\,\tau_{n-1}^{(\vartheta )} \, \frac{1 - \overline{\tau_{n-1}^{(\vartheta )}\alpha_{n-1}}}{1 - \tau_{n-1}^{(\vartheta )}\alpha_{n-1}}, \quad n \geq 1.
\] 
Having $\{\tau_{n}^{(\vartheta )}\}_{n=0}^{\infty}$ we define  $\{c_n^{(\vartheta )}\}_{n=1}^{\infty}$ and $\{d_{n+1}^{(\vartheta )}\}_{n=1}^{\infty}$ by
\begin{equation*}%\label{Eq-CoeffsTTRR-2}
\dsp c_{n}^{(\vartheta )} = \frac{-\Im (\tau_{n-1}^{(\vartheta )}\alpha_{n-1})} {1-\Re(\tau_{n-1}^{(\vartheta )}\alpha_{n-1})}, \quad g_{n}^{(\vartheta )} = \frac{1}{2} \frac{\big|1 - \tau_{n-1}^{(\vartheta )} \alpha_{n-1}\big|^2}{\big[1 - \Re(\tau_{n-1}^{(\vartheta )}\alpha_{n-1})\big]}, \quad n \geq 1,  
\end{equation*}
and let
\begin{equation*}%\label{Eq-CoeffsTTRR-2}
 d_{n+1}^{(\vartheta )} = (1-g_{n}^{(\vartheta )})g_{n+1}^{(\vartheta )}.
\end{equation*}
Finally, for $x\in [-1,1]$ we set
\[
\mathfrak d_{n+1}^{(\vartheta)}(x) = \frac{d_{n+1}^{(\vartheta)}}{\big(x - c_{n}^{(\vartheta)}\sqrt{1-x^2}\big)\big(x - c_{n+1}^{(\vartheta)}\sqrt{1-x^2}\big) }. 
\]

\begin{theo} \label{Thm-Restriction-3-for-cn-dn}
	Let  $\{\alpha_n\}_{n=0}^{\infty}$  be the Verblunsky coefficients for a  nontrivial positive measure  $\mu$ on the unit circle, and let $0 <  \vartheta_2 - \vartheta_1 \leq 2 \pi$ (with $\vartheta_2 > 0$). 
With the definitions introduced above, set 
\[ 
  \mathfrak{m}_0^{(\vartheta_1, \vartheta_2)} = 0\quad \text{and} \quad 
	\dsp  \mathfrak{m}_n^{(\vartheta_1, \vartheta_2)} = \frac{\mathfrak d_{n+1}^{(\vartheta_2)}\big(x(\vartheta_1,\vartheta_2)\big)}{1-\mathfrak{m}_{n-1}^{(\vartheta_1, \vartheta_2)}}, \quad n \geq 1,
\]
with
\[
   x(\vartheta_1,\vartheta_2):=\cos\big((2\pi+\vartheta_1-\vartheta_2)/2\big).
\]
	Then 
\[
	\supp \mu  \cap \mathcal{A}(\vartheta_1, \vartheta_2)  =\emptyset
\]
if and only if  
\[
	  \cot\big((2\pi+\vartheta_1-\vartheta_2)/2\big) < c_{1}^{(\vartheta )}\quad \mbox{and} \quad 0 < \mathfrak{m}_n^{(\vartheta_1, \vartheta_2)} < 1, \ \ n \geq 1.
\]
\end{theo}

The paper is organized as follows. Sections \ref{sec:CDandPositiveCS} and \ref{Sec-PrelimResults-3} are a brief summary of main results from \cite{BracMcCabPerezRanga-MCOM2015}, \cite{Castillo-Costa-Ranga-Veronese-JAT2014}, \cite{Costa-Felix-Ranga-JAT2013} and \cite{Dimitar-Ranga-MN2013}, necessary  for a better understanding of the rest of the manuscript.  Section \ref{Sec-BEZeros} describes a simple techniques for finding bounds for the extreme zeros of the polynomials $R_n$ from the coefficients $\{c_n\}$ and $\{d_{n+1}\}$, while Section \ref{Sec-BSupportM} relates these bounds with the support of the measure. Section \ref{Sec-SORPCS} contains some further results on scaling sequences, used in the construction of Example 3 in Section \ref{Sec-Examples}.  Examples 1 and 2 from Section \ref{Sec-Examples} illustrate with explicit formulas the tightness of the bounds obtained in this paper.

%%%%%%%%%%%%%%%%%%%%%%%%%%%%%%%%%%%%%%%%%%%%%%%%%%
\setcounter{equation}{0}
\section{CD kernel POPUC and positive chain sequences} \label{sec:CDandPositiveCS}
 
In this section we gather some of the main results from \cite{Costa-Felix-Ranga-JAT2013}, providing a necessary background for the rest of the paper.

The CD kernel $K_{n}(w,z)$ of a non-trivial probability measure $\mu$ on $\T$ was introduced in \eqref{def:CDkernel}. For $w\in \T$ fixed, this is a polynomial of degree $n+1$ in $z$, which can be renormalized monic by defining 
\begin{equation*} \label{Eq-monic-CD-Kernel-POPUC}
P_n(w;\,z) :=  \frac{\kappa_{n+1}^{-2}\,\overline{w}}{\overline{\Phi_{n+1}(w)}} \frac{K_n(w,\,z)}{1 + \tau_{n+1}(w) \alpha_{n}} = \frac{z\Phi_{n}(z) - w\tau_{n}(w) \Phi_{n}^{\ast}(z)}{z-w}, \quad n\geq 0 ,
\end{equation*}
where $\{\alpha_n\}_{n\geq 0}$ are the Verblunsky coefficients for $\mu$, and $\tau_n(w)$ is given by \eqref{defTau}.

In what follows we fix $w=1$, and set $R_0(z)=1$, 
\begin{equation}\label{Eq-CDKernel-POPUC}
R_{n}(z) = \left(\prod_{j=1}^{n} \frac{  1- \tau_{j-1}\alpha_{j-1} }{   1-\Re(\tau_{j-1}\alpha_{j-1}) }\right) \,P_{n}(1;\,z),  \quad n \geq 1,
\end{equation}
with $\tau_j = \tau_j(1)$, $j \geq 0$. It is important to notice that if we know the Verblunsky coefficients for $\mu$, $\tau_j$'s can be easily computed recursively by
\begin{equation}\label{recurrenceC}
\tau_{n} = \frac{\tau_{n-1} - \overline{\alpha}_{n-1}}{1 - \tau_{n-1}\alpha_{n-1}},  \quad n \geq 1,
\end{equation}
starting with $\tau_0=1$.

It turns out that the sequence  of polynomials $\{R_{n}\}_{n\geq 0}$ in \eqref{Eq-CDKernel-POPUC} satisfies the three term recurrence formula (\ref{Eq-TTRR-Rn}). The coefficients $c_n$ are explicitly given in terms of $\alpha_n$'s by
\begin{equation}\label{cSeq}
\dsp c_{n} = \frac{-\Im (\tau_{n-1}\alpha_{n-1})} {1-\Re(\tau_{n-1}\alpha_{n-1})} \in \R, \quad n\geq 1.
\end{equation}
It was mentioned also that $\{d_{n+1}\}_{n=1}^{\infty}$ is a positive chain sequence: it satisfies \eqref{def:CS} with the parameter sequence  $\{g_{n}\}_{n=1}^{\infty}$,   
\begin{equation}\label{parametricSeq}
	  g_{n} = \frac{1}{2} \frac{\big|1 - \tau_{n-1} \alpha_{n-1}\big|^2}{\big[1 - \Re(\tau_{n-1}\alpha_{n-1})\big]}, \quad n \geq 1.
\end{equation}

Polynomials $R_n$, $n \geq 1$ satisfy 
\begin{equation*} 
\int_{\mathbb{T}} \zeta^{-n+j} R_{n}(\zeta) (1 - \zeta) d \mu(\zeta) = 0, \quad  0 \leq j \leq n-1.
\end{equation*}
Observe that this relation remains unaltered by an addition to $\mu$ of a Dirac measure (or mass point) at $\zeta=1$ (so-called \textit{Uvarov transformation}). This fact induces us to expect that $\{R_n(z)\}_{n=0}^{\infty}$ are invariant under such a transformation. This is the case indeed. Assume that $\mu(\{1\})=a\in [0,1)$; then 
\begin{equation} \label{Eq-Family-of-Measures}
 \mu(t;\cdot)=\frac{1-t}{1-a}\mu  +  \frac{t - a}{1-a} \delta_1 , \quad t\in [0,1),
\end{equation}
is a nontrivial positive probability Borel measure on $\T$ such that $\mu(t;\{1\})=t$. Let us denote by  $\{\Phi_n(t; z)\}_{n=0}^{\infty}$ and $\{\alpha_n^{(t)}\}_{n=0}^{\infty}$ the corresponding monic OPUC and Verblunsky coefficents for the measure $\mu(t;.      )$; obviously, they depend on the parameter $t$.

A remarkable fact, first observed in \cite{Costa-Felix-Ranga-JAT2013},  is that the sequence of polynomials $\{R_n(z)\}_{n=0}^{\infty}$ (and thus, $\{P_n(1; z)\}_{n=0}^{\infty}$), and the sequences of constants $\{\tau_n\}_{n=0}^{\infty}$, $\{c_n\}_{n=1}^{\infty}$ and $\{d_{n+1}\}_{n=1}^{\infty}$, do not depend on $t$ and are exactly the same for all measures $\{\mu(t; \cdot) , \,0 \leq t < 1\}$. 
This is not the case  however of the parameter  sequence $g_n$, defined in \eqref{parametricSeq}: if 
\[
    d_{n+1} = (1-g_n^{(t)})\, g_{n+1}^{(t)} \quad \mbox{and} \quad    g_{n}^{(t)} = \frac{1}{2} \frac{\big|1 - \tau_{n-1} \alpha_{n-1}^{(t)}\big|^2}{ 1 - \Re(\tau_{n-1}\alpha_{n-1}^{(t)}) }, \quad n \geq 1,
\]
then  the sequence $\{g_{n+1}\}_{n=0}^{\infty} = \{g_{n+1}^{(t)}\}_{n=0}^{\infty}$ does depend on the value of $t$.  Thus, $\{g_{n+1}^{(t)}\}_{n=0}^{\infty}$ for each $t$ represents a different parameter sequence of the positive chain sequence $\{d_{n+1}\}_{n=1}^{\infty}$, which is necessarily non {SP}. It is easy to show that 
$$
g_{1}^{(t)}=\frac{1}{2} \frac{\big|1 -   \alpha_{0}^{(t)}\big|^2}{\big[1 - \Re( \alpha_{0}^{(t)})\big]}
$$
is decreasing in $t$, so that   $\{g_{n+1}^{(0)}\}_{n=0}^{\infty} = \{M_{n+1}\}_{n=0}^{\infty}$ is the maximal parameter sequence of $\{d_{n+1}\}_{n=1}^{\infty}$. 

The polynomials $R_n$ satisfy the orthogonality relation
\begin{equation} \label{Eq-Rn-Ortogonality}
     \int_{\mathbb{T}} \zeta^{-n+j} R_{n}(\zeta) (1 - \zeta) d \mu(t;\zeta) = 
     \left\{
      \begin{array}{ll}
        0, & 0 \leq j \leq n-1, \\
        \gamma_{n}^{(t)}, & j =n,
      \end{array}
     \right.
\end{equation}
with
\[
    \gamma_{n}^{(t)} = (1- \overline{\tau_n \alpha_n^{(t)}}) \prod_{k=0}^{n-1} \frac{(1-\overline{\tau_k \alpha_k^{(t)}})(1 - |\alpha_k^{(t)}|^2)}{1 - \Re(\tau_{k} \alpha_{k}^{(t)})} = \frac{2(1-t)M_{1}}{1+ic_1} \prod_{k=1}^{n} \frac{4 d_{k+1}}{1 + i c_{k+1}}. 
\]

Sequences $\{c_n\}_{n=1}^{\infty}$ and $\{d_{n+1}\}_{n=1}^{\infty}$ give a parametrization of the family \linebreak $\{\mu(t; \cdot), \, 0 \leq t < 1\}$, in the same vein as $\{\alpha_n\}_{n\geq 0}$ parametrize $\mu$. For convenience, we call it the \textit{$(c_n, d_n)$-parametrization} of this family of measures. This result can be seen also as a unit circle analogue of the classical Favard theorem on $\R$ (see \cite{Castillo-Costa-Ranga-Veronese-JAT2014}).   Indeed, from the results established in \cite{Castillo-Costa-Ranga-Veronese-JAT2014} and \cite{Costa-Felix-Ranga-JAT2013}, if $\{c_n\}_{n=1}^{\infty}$  is any real sequence, and  $\{d_{n+1}\}_{n=1}^{\infty}$ is  a non SP positive chain sequence, then for each $t\in [0,1]$ there exists a unique positive probability Borel measure $\mu(t; \cdot)$ on $\T$ and the corresponding sequence of its Verblunsky coefficients $\{\alpha_{n}^{(t)}\}_{n\geq 0}$ such that $c_n$'s and $d_n$'s are calculated from them by formulas \eqref{cSeq} and \eqref{parametricSeq}, and $\mu(t;\{1\})=t$. Furthermore, if we denote by  $\{M_{n}\}_{n=1}^{\infty}$ the maximal parameter sequence of $\{d_{n+1}\}_{n=1}^{\infty}$, and by $\{m_n^{(t)}\}_{n=0}^{\infty}$ the  minimal parameter sequence of the positive chain sequence $\{d_{n}\}_{n=1}^{\infty}$ augmented by the term $d_1=(1-t)M_{1}$, then 
\begin{equation} \label{Eq-Recovery1-VErbCoeffs}
\alpha_{n-1}^{(t)} = \frac{1}{\tau_{n-1}}\,\frac{1-2m_{n}^{(t)} - i c_n}{1 - i c_n},  \quad n \geq 1,
\end{equation}
where the sequence  $\{\tau_n\}_{n=0}^{\infty}$ is calculated by (\ref{recurrenceC}) and satisfies
\[
\tau_{0}= 1, \quad \tau_{n} =  \tau_{n-1}\frac{1-ic_n}{1+ic_n}, \quad n \geq 1.
\]

One of the consequences of the result established in \cite[Corollary 4.19]{CantMoraVela-JAT2006} (see also the proof of Theorem 4.1 in \cite{Castillo-Costa-Ranga-Veronese-JAT2014}) is the following theorem:

\begin{theo}  \label{ThmA-Support}
	If for some $0 \leq \vartheta_1 < \vartheta_2 \leq 2\pi$ and for all sufficiently large $n$, the zeros of $R_n$ belong to the open arc  $\mathcal{A}(\vartheta_1, \vartheta_2)$, then for all $0 \leq t< 1$,
	\begin{equation}\label{emptyarcs}
\supp \mu(t;\cdot) \cap \mathcal{A}(0, \vartheta_1) = \supp \mu(t;\cdot) \cap \mathcal{A}(\vartheta_2, 2\pi)=\emptyset.
	\end{equation}

\end{theo}

In particular, for the extreme zeros $z_{n,1} = e^{i\theta_{n,1}}$ and $z_{n,n} = e^{i\theta_{n,n}}$  of $R_n$,  we have that if 
$$
\vartheta_1 :=\lim_{n \to \infty} \theta_{n,1}, \quad \text{and} \quad  \vartheta_2: =\lim_{n \to \infty} \theta_{n,n},
$$
then \eqref{emptyarcs} holds, and 
$$
e^{i \vartheta_1}, e^{i \vartheta_2} \in  \supp \mu(t;\cdot).
$$ 

\begin{remark}
 Because of the interlacing of the zeros of $R_n$, $n\geq 1$, ``all  sufficiently large $n$'' in Theorem \ref{ThmA-Support} is also equivalent to saying ``all $n$''.  
\end{remark}

Our next goal will be to obtaining bounds for the extreme zeros $z_{n,1} $ and $z_{n,n}$ directly from the sequences $\{c_n\}$ and $\{d_{n+1}\}$.

%%%%%%%%%%%%%%%%%%%%%%%%%%%%%%%%%%%%%%%%%%%%%%%%%%
\setcounter{equation}{0}
%\subsection{Related functions with simple zeros within  $(-1,1)$}   

\section{Delsarte-Genin transformation} \label{Sec-PrelimResults-3}

It is  much more convenient to study the zeros of $\{R_n\}$ by their transplantation to the real line by means of the  functions $\{\mathcal{W}_n\}$  defined by 
\begin{equation} \label{Eq-for-Wn}
    \mathcal{W}_n(x) = 2^{-n} e^{-in\theta/2} R_{n}(e^{i\theta}), \quad n \geq 0,
\end{equation}
where $x = \cos(\theta/2)$. The transformation $x = \cos(\theta/2) = (e^{i\theta/2}+e^{-i\theta/2})/2$, is known as the Delsarte-Genin transformation (see  \cite{Zhedanov-JAT98}, where the abbreviation  DG transformation is also used). 

The sequence of functions $\{\mathcal{W}_n\}_{n=0}^{\infty}$ satisfies (see \cite{BracMcCabPerezRanga-MCOM2015, Dimitar-Ranga-MN2013}) 
\begin{equation} \label{TTRR-for-Wn}
    \mathcal{W}_{n+1}(x) = \left(x - c_{n+1}\sqrt{1-x^2}\right)\mathcal{W}_{n}(x) - d_{n+1}\,\mathcal{W}_{n-1}(x), \quad n \geq 0,
\end{equation}
with $\mathcal{W}_{-1}(x) = 0$ and $\mathcal{W}_{0}(x) = 1$. % and $\mathcal{W}_{1}(x) = x - c_{1}\sqrt{1-x^2}$.  
Functions $\mathcal{W}_n$ bear a resemblance to standard orthogonal polynomials on the real line. For instance, for any $n \geq 1$, $\mathcal{W}_n$ has exactly $n$ distinct zeros $x_{n,j} = \cos(\theta_{n,j}/2)$, $j = 1,2 \ldots, n$, all in $(-1,1)$.  Another consequence of the three term recurrence formula (\ref{TTRR-for-Wn}) is the interlacing property 
\begin{equation}  \label{Eq-InterlacingZeros-Wn}
-1 < x_{n+1,n+1} < x_{n,n} < x_{n+1,n} < \cdots < x_{n,1} < x_{n+1,1} < 1, \quad n \geq 1,
\end{equation}
for the zeros of $\mathcal{W}_n$ and $\mathcal{W}_{n+1}$ (fact that was used in \cite{Dimitar-Ranga-MN2013} to prove the interlacing property (\ref{Eq-InterlacingZeros-Rn}) for the zeros of $R_n$ and $R_{n+1}$ on $\T$).  

In \cite{BracMcCabPerezRanga-MCOM2015}, the function $\mathcal{W}_n$ is considered as belonging to the simple field extension of the field of polynomials, generated by the adjunction of $\sqrt{1-x^2}$. More precisely, if $\mathbb{P}_n$ represents the polynomials of degree $\leq n$, we consider the  linear space of real functions $\Omega_n$, $n\ge 0$,  given by $\Omega_{0} = \mathbb{P}_{0}$ and 
$$
\Omega_n =\left\{B_n^{(0)}(x) + \sqrt{1-x^2} B_{n-1}^{(1)}(x):\, B_{k}^{(j)}\in \mathbb P_k\text{ and } B_k^{(j)}(-x) = (-1)^{k}B_k^{(j)}(x)  \right\}.
$$
Observe that if $\mathcal{F} \in \Omega_{n-2}$ then $\mathcal{F} \in  \Omega_{n}$, although $\mathcal{F} \notin  \Omega_{n-1}$.

Moreover (see \cite{DimIsmailRan-2012} and also \cite[Lemma 2.1]{BracMcCabPerezRanga-MCOM2015}), if $\mathcal{F}_n \in \Omega_n$ then the polynomial $Q_n$ defined by $r_n e^{-in\theta/2} Q_n(e^{i\theta}) = \mathcal{F}_n(\cos(\theta/2))$, $r_n\in \R$,  is self-reciprocal or conjugate reciprocal  (see \cite{Sinclair-Vaaler-2008}): it satisfies $Q_n^{\ast}(z) = z^{n}\overline{Q_n(1/\overline{z})} = Q_n(z)$.  
Notice that the polynomials $R_n$, given by \eqref{Eq-for-Wn}, share the same property.

With respect to  the zeros of $\mathcal{F}_n \in \Omega_n$ we know the following (see \cite[Thm.\,2.2]{BracMcCabPerezRanga-MCOM2015}):
\begin{itemize}
\item If $\mathcal{F}_n(1) \neq 0$, then the number of zeros of $\mathcal{F}_n$ in $(-1,1)$ can not exceed $n$.  

\item If $\mathcal{F}_n(1) = 0$, then $\mathcal{F}_n(-1) = 0$ and the number of zeros of $\mathcal{F}_n$ in $[-1,1]$ can not exceed $n+1$.
\end{itemize}
The straightforward identity  
\[ 
   e^{-i\theta/2}(e^{i\theta} - e^{i\theta_{j}}) = (1- e^{i\theta_j})\big[x - \cot(\theta_{j}/2) \sqrt{1-x^2}\big], \quad x = \cos(\theta/2),
\]
yields also the following factorization for $\mathcal{F}_n \in \Omega_n$:
\begin{lemma}  \label{Lema-Division} Let $\mathcal{F}_{n} \in \Omega_n$. With $1 \leq r \leq n$, let   $x_{j} = \cos(\theta_{j}/2)$, $j=1,2, \ldots, r$,  be zeros of $\mathcal{F}_{n}$ in $(-1,1)$.  Then 
\[
     \mathcal{F}_n(x) = \mathcal{G}_{n-r}(x) \prod_{j=1}^{r} (x - \beta_{j} \sqrt{1-x^2}), 
\]
where $\mathcal{G}_{n-r} \in \Omega_{n-r}$ and $\beta_{j} = \cot (\theta_{j}/2)$, $j = 1, 2 \ldots, r$. 
\end{lemma}

From this, we have in particular for the functions $\mathcal{W}_n$,  
\[
     \mathcal{W}_n(x) = \prod_{j=1}^{n} \frac{\big[1- \tau_{j-1}\alpha_{j-1}\big](1-z_{n,j})}{2 \left[1-\Re(\tau_{j-1}\alpha_{j-1})\right]} \, \prod_{j=1}^{n} \left(x - \beta_{n,j} \sqrt{1-x^2}\right), \quad n \geq 1,
\]
where $\beta_{n,j} = \cot(\theta_{n,j}/2)$ and $z_{n,j} = e^{i \theta_{n,j}}$, $j =1,2, \ldots, n$.

A direct consequence of (\ref{Eq-Rn-Ortogonality}) is the following orthogonality relation  for  $\{\mathcal{W}_n\}_{n=0}^{\infty}$:
\[
    \int_{-1}^{1} (x + i \sqrt{1-x^2})^{-n+1+2j} \mathcal{W}_n(x) \, d \psi(x) = 
       \left\{
      \begin{array}{ll}
        0, & 0 \leq j \leq n-1, \\
        i\, 2^{-n-1} \gamma_{n}^{(t)}, & j =n,
      \end{array}
     \right.
\]
where $d\psi(x) = - \sqrt{1-x^2} \, d\mu(e^{i2\arccos(x)})$, which in turn yields (see \cite[Theorem 4.1 and Corollary 4.1.3]{BracMcCabPerezRanga-MCOM2015}) that $n,m = 0, 1, 2, \dots \ $,
\begin{equation} \label{Orthogonality-for-Wm}
  \begin{array}{l}
    \dsp \int_{-1}^{1} \mathcal{W}_{2n}(x)\,\mathcal{W}_{2m}(x)\, \sqrt{1-x^2}\,d\psi(x) = \chi_{2n}\,\delta_{n,m} ,  \\[3ex]
    \dsp  \int_{-1}^{1} \mathcal{W}_{2n+1}(x)\,\mathcal{W}_{2m+1}(x)\,\sqrt{1-x^2}\,d\psi(x) =  \chi_{2n+1} \,\delta_{n,m},\\[3ex]
     \dsp  \int_{-1}^{1} \mathcal{W}_{2n}(x)\,\mathcal{W}_{2m+1}(x)\, d\psi(x) = 0,
  \end{array}
\end{equation}
where 
\[
     \chi_{0} = \frac{1}{2} [1 - \Re(\alpha_0)] \quad \mbox{and} \quad \chi_{n} =   d_{n+1} \frac{1+c_{n}^2}{1+c_{n+1}^2} \chi_{n-1}, \ \  n\geq 1. 
\]

 We can also express \eqref{Orthogonality-for-Wm} by (see \cite[Corollary 4.1.1]{BracMcCabPerezRanga-MCOM2015})
 \begin{equation}\label{eqOrthogonalityC1}
  \dsp \int_{-1}^{1} \mathcal{F}(x)\,\mathcal{W}_{n}(x)\,d\psi(x) = 0 \quad \mbox{whenever} \quad \mathcal{F} \in \Omega_{n-1}.
 \end{equation}

The following result, related to Theorem~\ref{ThmA-Support},  illustrates the advantage of working with the sequence $\mathcal{W}_n$ (and thus, $R_n$): 
\begin{theo} \label{Thm-Zeros-in-Gaps}
Let $\mu$ be a nontrivial positive measure on $\T$, with $R_n$ and $\mathcal{W}_n$  given by $(\ref{Eq-CDKernel-POPUC})$ and  $(\ref{Eq-for-Wn})$, respectively.  Assume $0 < \vartheta_1 < \vartheta_2 < 2\pi$. Then with the notation \eqref{notationArcs}:
\begin{enumerate}
\item [a)] If  $\supp \mu\cap \mathcal{A}(0, \vartheta_1)=\emptyset$ then the function $\mathcal{W}_n$ does not have any zeros in the interval $[\cos(\vartheta_1/2), 1]$ and, equivalently, the polynomial $R_n$ does not vanish on the closed arc $\mathcal{A}[0, \vartheta_1]$. \\[-2ex]

\item [b)] If  $\supp \mu\cap \mathcal{A}(\vartheta_2, 2\pi)=\emptyset$   then the function $\mathcal{W}_n$ does not have any zeros in the interval $[-1, \cos(\vartheta_2/2)]$  and, equivalently, the polynomial $R_n$ does not vanish on the closed arc $\mathcal{A}[\vartheta_2, 2\pi]$. \\[-2ex]

\item [c)] If $\supp \mu\cap \mathcal{A}(\vartheta_1, \vartheta_2)=\emptyset$  then the function  $\mathcal{W}_n$ has at most one zero  in the interval $[\cos(\vartheta_2/2), \cos(\vartheta_1/2)]$ and, equivalently, the polynomial $R_n$ has at most one zero in the closed arc $\mathcal{A}[\vartheta_1, \vartheta_2]$. 

\end{enumerate}

\end{theo}

\noindent {\bf Proof.}  If $\mathcal{A}(0, \vartheta_1)$ is a gap in the support of $\mu$   then the support of $\psi$, where $d \psi(x) = - \sqrt{1-x^2} \, d\mu(e^{i2\arccos(x)})$, stays within the interval $[-1, \cos(\vartheta_1/2)]$. Suppose that $\mathcal{W}_n$ has the zero $y = \cos(\vartheta/2)$ in $[\cos(\vartheta_1/2), 1]$. By Lemma \ref{Lema-Division}, 
\[
   \mathcal{W}_n(x) =  (x - \beta \sqrt{1-x^2}) \, \mathcal{G}_{n-1}(x), 
\]
where $\mathcal{G}_{n-1} \in \Omega_{n-1}$ and $\beta = \cot(\vartheta/2)$.  Hence, by \eqref{eqOrthogonalityC1}, 
\[ 
   I = \int_{-1}^{1} (\mathcal{G}_{n-1}(x))^2 (x - \beta \sqrt{1-x^2}) \,d \psi(x) 
     = \int_{-1}^{1} \mathcal{G}_{n-1}(x) W_{n}(x) \,d \psi(x) = 0.
\]  
On the other hand, since the support of $\psi$ stays within the interval $[-1, \cos(\vartheta_1/2)]$ and since $(x - \beta \sqrt{1-x^2}) < 0$ in $[-1,\cos(\vartheta_1/2))$, one must have 
\[
    I = \int_{-1}^{\cos(\vartheta_1/2)} (\mathcal{G}_{n-1}(x))^2 (x - \beta \sqrt{1-x^2}) \,d \psi(x) < 0. 
\]
That is, the existence of a zero of $\mathcal{W}_n$ in $[\cos(\vartheta_1/2), 1]$ leads to a contradiction. This proves \emph{a)}; the proof of \emph{b)} is similar.

Now we turn to the statement \emph{c)}. If   $\mathcal{A}(\vartheta_1, \vartheta_2)$ is a gap in  $\supp \mu$, then the support of $\psi$ stays outside of  the subinterval $(\cos(\vartheta_2/2), \cos(\vartheta_1/2))$. Suppose that $y_1 = \cos(\vartheta^{(1)}/2)$ and $y_2 = \cos(\vartheta^{(2)}/2)$  are zeros of $\mathcal{W}_n$ in $[\cos(\vartheta_2/2), \cos(\vartheta_1/2)]$. By Lemma \ref{Lema-Division},
\[
   \mathcal{W}_n(x) =  (x - \beta_1 \sqrt{1-x^2}) (x - \beta_2\sqrt{1-x^2})\, \mathcal{G}_{n-2}(x), 
\]
where $\mathcal{G}_{n-2} \in \Omega_{n-2}$, $\beta_1 = \cot(\vartheta^{(1)}/2)$ and $\beta_2 = \cot(\vartheta^{(2)}/2)$. Since $\sqrt{1-x^2} \mathcal{G}_{n-2}(x) \in   \Omega_{n-1}$ we have by \eqref{eqOrthogonalityC1} 
\[ 
   \tilde{I} = \int_{-1}^{1} (\mathcal{G}_{n-2}(x))^2 (x - \beta_1 \sqrt{1-x^2}) (x - \beta_2 \sqrt{1-x^2}) \sqrt{1-x^2}\,d \psi(x) 
     = 0.
\]  
On the other hand, since $(x - \beta_1 \sqrt{1-x^2}) (x - \beta_2 \sqrt{1-x^2})> 0$ outside the subinterval $[\cos(\vartheta_2/2), \cos(\vartheta_1/2)]$, one must have 
\[  
  \begin{array}l 
    \dsp \tilde{I} = \int_{-1}^{\cos(\vartheta_2/2)} (\mathcal{G}_{n-1}(x))^2 (x - \beta_1 \sqrt{1-x^2}) (x - \beta_2\sqrt{1-x^2}) \sqrt{1-x^2}\,d \psi(x) \\[2ex]
    \dsp \qquad \quad + \int_{\cos(\vartheta_1/2)}^{1} (\mathcal{G}_{n-1}(x))^2 (x - \beta_1 \sqrt{1-x^2}) (x - \beta_2\sqrt{1-x^2}) \sqrt{1-x^2}\,d \psi(x)> 0. 
  \end{array}
\]
That is, the existence of more than one zero of $\mathcal{W}_n$ in $[\cos(\vartheta_2/2), \cos(\vartheta_1/2)]$ leads to a contradiction.   \eProof

\begin{remark}
Statement \emph{c)} of Theorem \ref{Thm-Zeros-in-Gaps} was previously established in \cite{CantMoraVela-2002} and \cite{Golinskii-2002}, by different techniques.  
\end{remark}

%%%%%%%%%%%%%%%%%%%%%%%%%%%%%%%%%%%%%%%%%%%%%%%%%%%%%%%%%%%%%%%
\setcounter{equation}{0} %\section{Bounds for the extreme zeros}
\section{Extreme zeros from the $(c_n, d_n)$-parametrization}  \label{Sec-BEZeros}

We first look at the information about the zeros of  $\mathcal{W}_n$ and $R_n$ that can be extracted from the coefficients $\{c_n\}_{n=1}^{\infty}$ and $\{d_{n+1}\}_{n=1}^{\infty}$ of the recurrence formula \eqref{TTRR-for-Wn}. We are particularly interested in the smallest, $x_{n,n}$, and the largest, $x_{n,1} $, zero of $\mathcal{W}_n$, see \eqref{Eq-InterlacingZeros-Wn}. The   following bounds were already established in \cite{Dimitar-Ranga-MN2013}:
\begin{itemize}
\item if \ $c_k > 0$ for $1 \leq k \leq n$ \ and \ $\dsp \check{c}_{n} = \min_{1 \leq k \leq n} c_k$, \ then \ $\dsp -\frac{1}{\sqrt{1+\check{c}_{n}^2}} < x_{n,n}$;

\item if \ $c_k < 0$ for $1 \leq k \leq n$ \ and \ $\dsp \hat{c}_{n} = \max_{1 \leq k \leq n} c_k$, \ then \ $\dsp x_{n,1} < \frac{1}{\sqrt{1+\hat{c}_{n}^2}}$.
\end{itemize}

Observe that these bounds depend only on the sequence $\{c_n\}$; the following results extend and improve them by taking into account also the chain sequence $\{d_{n+1}\}$. Given these two sequences, we define for $x\in [-1,1]$
\begin{equation}\label{definitionFD}
 \fd_{n+1}(x)  : =  \frac{d_{n+1}} 
 {(x - c_{n}\sqrt{1-x^2})(x - c_{n+1}\sqrt{1-x^2})} .
\end{equation}

\begin{theo} \label{Thm-General-Zero-Bounds-Wn} Let $\mathcal{W}_n$ be given by the  three term recurrence formula $(\ref{TTRR-for-Wn})$, where 
$\{c_n\}_{n=1}^{\infty}$ is a real sequence and  $\{d_{n+1}\}_{n=1}^{\infty}$ is a positive chain sequence.   Then, for $N \geq 2$, the zeros of $\mathcal{W}_n$, $1 \leq n \leq N$, all belong to $(A,B) \subseteq (-1,1)$ if, and only if, 
\begin{enumerate}
\item[a)] $\dsp \frac{A}{\sqrt{1-A^2}} < c_1 <\frac{B}{\sqrt{1-B^2}}$,

\item[b)] $\{\fd_{n+1}(x)\big\}_{n=1}^{N-1}$ is a finite positive chain sequence for $ x=A$  and $ x=B$.
\end{enumerate} 
\end{theo}

\noindent {\em Proof}. Observe that  $A-c_1\sqrt{1-A^2} < 0$ and $\fd_{n+1}(A) > 0$, $n \geq 1$,  imply that  $A - c_n\sqrt{1-A^2} < 0$, $n = 1,2, \dots, N$. In the same vein,  $B-c_1\sqrt{1-B^2} > 0$ and $\fd_{n+1}(B) > 0$, $n \geq 1$,  imply  $B - c_n\sqrt{1-B^2} > 0$, $n = 1,2, \dots, N$. Thus, condition a) in fact means that
$$
\dsp \frac{A}{\sqrt{1-A^2}} < c_n <\frac{B}{\sqrt{1-B^2}} \quad n=1, \dots, N.
$$
Furthermore, from (\ref{TTRR-for-Wn}) that $\mathcal{W}_{n} \in C^2(-1,1)$,  
\begin{equation}  \label{Eq-RR1-for-Wn}
   \frac{\mathcal{W}_{n+1}(x)}{\mathcal{W}_{n}(x)} + d_{n+1} \frac{\mathcal{W}_{n-1}(x)}{\mathcal{W}_{n}(x)} = x-c_{n+1} \sqrt{1-x^2}, 
\end{equation}
and 
\begin{equation} \label{Eq-ChainRR-for-Wn}
  \begin{array}l
      \dsp  \frac{\mathcal{W}_{n}(x)}{(x-c_n \sqrt{1-x^2})\mathcal{W}_{n-1}(x)} \left(1 - \frac{\mathcal{W}_{n+1}(x)}{(x-c_{n+1} \sqrt{1-x^2})\mathcal{W}_{n}(x)}\right) = \fd_{n+1}(x),  
   \end{array}
\end{equation}
for $n \geq 1$. Using that $\mathcal{W}_{1}(1)/\mathcal{W}_{0}(1) = 1$ and that $\{d_{n+1}\}_{n=1}^{\infty}$ is a positive chain sequence, it follows that  $(-1)^{n-1}\,\mathcal{W}_{n-1}(-1) = \mathcal{W}_{n-1}(1) > 0$, $n \geq 1$. 

We  prove the theorem for the left bound $A$. Assume that the zeros of $\mathcal{W}_{N}$ are all greater than $A$. Then, by the interlacing property (\ref{Eq-InterlacingZeros-Wn}), $\mathcal{W}_{n}(A)/\mathcal{W}_{n-1}(A) < 0$, $1 \leq n \leq N$. Hence, from $\mathcal{W}_{1}(A) = A - c_1 \sqrt{1-A^2}$ and from  (\ref{Eq-RR1-for-Wn}), $A - c_n \sqrt{1-A^2} < 0$, $1 \leq n \leq N$, establishing the lower bound for $c_n$ in a). %Thus, part ($a$) of the theorem must hold with respect to the point $A$.  

Now, since $\mathcal{W}_{n}(A)/[(A - c_n \sqrt{1-A^2})\mathcal{W}_{n-1}(A)] > 0$ for $1 \leq n \leq N$ and 
\[
    \fd_{n+1}(A) = \frac{d_{n+1}} 
               {(A - c_{n}\sqrt{1-A^2})(A - c_{n+1}\sqrt{1-A^2})} > 0, \quad 1 \leq n \leq N-1,
\]  
from (\ref{Eq-ChainRR-for-Wn}) we have
\[
   \mathcal{W}_{n+1}(A)/[(A - c_{n+1} \sqrt{1-A^2})\mathcal{W}_{n}(A)] \ < 1 \quad  \mbox{for} \quad  1 \leq n \leq N-1.
\]

An immediate consequence of (\ref{Eq-ChainRR-for-Wn}) is that $\{\fd_{n+1}(A)\}_{n=1}^{N-1}$ is a finite positive chain sequence. Thus,  part \textit{b)} of the theorem  holds for the point $A$. Notice that  $\{1-\mathcal{W}_{n+1}(A)/[(A - c_{n+1} \sqrt{1-A^2})\mathcal{W}_{n}(A)]\}_{n=0}^{N-1}$ is the minimal parameter sequence for $\{\fd_{n+1}(A)\}_{n=1}^{N-1}$.

Now we prove the reciprocal statement, again at  the point $A$, assuming that both \textit{a)} and \textit{b)}, for $x=A$, hold.

By the  assumption \textit{a)},  $\mathcal{W}_1(A) =  A - c_1\sqrt{1-A^2} < 0$ (which means $A < x_{1,1}$)  and $A - c_{n+1}\sqrt{1-A^2} < 0$ for $n = 1, 2, \ldots, N-1$. Using \textit{b)} we have that 
\[
      \left\{1 - \frac{\mathcal{W}_{n+1}(A)}{(A-c_{n+1} \sqrt{1-A^2})\mathcal{W}_{n}(A)}\right\}_{n=0}^{N-1} 
\]
is the minimal parameter sequence of the finite positive chain sequence $\{\fd_{n+1}(A)\}_{n=1}^{N-1}$.  Hence, 
\[
     0 < \frac{\mathcal{W}_{n+1}(A)}{(A-c_{n+1} \sqrt{1-A^2})\mathcal{W}_{n}(A)} < 1, \quad n = 1, 2, \ldots, N-1.  
\]
From this,  $(-1)^{n}\, \mathcal{W}_{n}(A) > 0$, $n = 1, 2, \ldots, N$.

From the interlacing property (\ref{Eq-InterlacingZeros-Wn})  for the zeros $x_{n,j}$, $j = 1,2 \ldots, n$,  of $\mathcal{W}_n$ observe that  $(-1)^j\mathcal{W}_{n}(x_{n-1,j}) > 0$ for $j =1,2, \ldots, n-1$ and $n \geq 2$. Hence, in particular,
\[
    (-1)^{n}\, \mathcal{W}_{n}(x_{n-1,n-1}) < 0, \quad n \geq 2.       
\]
Now the inequality $A < x_{N,N}$ easily follows by induction. 

This completes the proof of the theorem with respect to the point $A$; the proof for the point $B$ is similar. \eProof

With further restrictions on the coefficients $c_n$, $n \geq 1$, we get the following corollary of Theorem \ref{Thm-General-Zero-Bounds-Wn}: 
\begin{coro} \label{Coro-General2-Zero-Bounds-Wn} Let $\mathcal{W}_n$ be given by the  three term recurrence formula $(\ref{TTRR-for-Wn})$, where the sequences 
$\{c_n\}_{n=1}^{\infty}$ and  $\{d_{n+1}\}_{n=1}^{\infty}$ satisfy the conditions of  Theorem  \ref{Thm-General-Zero-Bounds-Wn}, for some even $N \geq 2$. Assume that there exist $C$ and $D$, $A < C < D < B$, such that for $\ell\in \{0, 1\}$  the elements of the sequence $\{c_n\}_{n=1}^{\infty}$ also satisfy the additional property 
\begin{equation} \label{Eq-odd-even-coefficient-bounds}
    \frac{A}{\sqrt{1-A^2}} < c_{2n-\ell} < \frac{C}{\sqrt{1-C^2}} <\frac{D}{\sqrt{1-D^2}} < c_{2n-1+\ell} <\frac{B}{\sqrt{1-B^2}}, 
\end{equation}
for $1 \leq n \leq  N/2$. 

Then the zeros of $\mathcal{W}_n$, $1 \leq n \leq N$, stay within  $(A,C) \cup (D,B)$.  
\end{coro}

\noindent {\em Proof}.  From Theorem \ref{Thm-General-Zero-Bounds-Wn} the zeros of $W_n$, $1 \leq n \leq N$ are all inside $(A,B)$. Since $W_1(x) = x - c_1 \sqrt{1-x^2}$, from (\ref{Eq-odd-even-coefficient-bounds}), 
\[
      (-1)^{\ell+1}W_1(C)  > 0  \quad \mbox{and} \quad (-1)^{\ell+1}W_1(D)  > 0.
\]
Thus, the single zero of $W_1$ stays within  $(A,C) \cup (D,B)$. 

Since 
\[
      (-1)^{\ell}[C - c_2\sqrt{1-C^2}] > 0 \quad \mbox{and} 
      \quad (-1)^{\ell}[D - c_2\sqrt{1-D^2}]  > 0,
\]
we now obtain from $(\ref{TTRR-for-Wn})$ that 
\[
  \begin{array}l
     (-1)^{1} W_2(C) = (-1)^{\ell}[C - c_2\sqrt{1-C^2}] (-1)^{\ell+1}W_1(C) - (-1)^{1} d_2 > 0, \\[2ex]
     (-1)^{1} W_2(D) = (-1)^{\ell}[C - c_2\sqrt{1-D^2}] (-1)^{\ell+1}W_1(D) - (-1)^{1} d_2 > 0.
  \end{array}
\]
Thus, together with the fact that the zeros of $W_1$ and $W_2$ interlace, we conclude that the zeros of $W_2$ stay within  $(A,C) \cup (D,B)$. 

We continue the proof by induction. Suppose that for a $k \geq 1$ the zeros of $W_{2k-1}$ and $W_{2k}$ stay within  $(A,C) \cup (D,B)$ and that 
\[
     (-1)^{\ell+k}W_{2k-1}(C) > 0,  \quad (-1)^{\ell+k}W_{2k-1}(D) > 0
\]
and 
\[
     (-1)^{k}W_{2k}(C) > 0,  \quad (-1)^{k}W_{2k}(D) > 0.
\]
Since 
\[
      (-1)^{\ell+1}(C - c_{2k+1}\sqrt{1-C^2}) > 0 \quad \mbox{and} 
      \quad (-1)^{\ell+1}(D - c_{2k+1}\sqrt{1-D^2}) > 0,
\]
we obtain from  
\[
  \begin{array}l
     (-1)^{\ell+k+1}W_{2k+1}(C) = (-1)^{\ell+k+1}\big[(C - c_{2k+1}\sqrt{1-C^2}) W_{2k}(C) - d_{2k+1} W_{2k-1}(C)\big], \\[2ex]
     (-1)^{\ell+k+1}W_{2k+1}(D) = (-1)^{\ell+k+1}\big[(D - c_{2k+1}\sqrt{1-D^2}) W_{2k}(D) - d_{2k+1} W_{2k-1}(D)\big], 
  \end{array}
\]
that 
\[
  \begin{array}l
     (-1)^{\ell+k+1}W_{2k+1}(C)  > 0 \quad \mbox{and} \quad 
     (-1)^{\ell+k+1}W_{2k+1}(D) > 0.
  \end{array}
\]
Thus, from the interlacing of the zeros, the zeros of $W_{2k+1}$  stay within  $(A,C) \cup (D,B)$ also. Now, continuing with  
\[
      (-1)^{\ell}(C - c_{2k+2}\sqrt{1-C^2}) > 0 \quad \mbox{and} 
      \quad (-1)^{\ell}(D - c_{2k+2}\sqrt{1-D^2}) > 0,
\]
we obtain from  
\[
  \begin{array}l
     (-1)^{k+1}W_{2k+2}(C) = (-1)^{2\ell+ k+1}\big[(C - c_{2k+2}\sqrt{1-C^2}) W_{2k+1}(C) - d_{2k+2}W_{2k}(C)\big], \\[1ex]
     (-1)^{k+1} W_{2k+2}(D) = (-1)^{2\ell+ k+1}\big[(D - c_{2k+2}\sqrt{1-D^2}) W_{2k+1}(D) - d_{2k+2}W_{2k}(D)\big],
  \end{array}
\]
that 
\[
  \begin{array}l
     (-1)^{k+1}W_{2k+2}(C)  > 0 \quad \mbox{and} \quad 
     (-1)^{k+1}W_{2k+2}(D) > 0.
  \end{array}
\]
From this the zeros of $W_{2k+2}$ also stay within  $(A,C) \cup (D,B)$. Thus, the proof follows by induction.  \hfil \eProof

We now apply Theorem \ref{Thm-General-Zero-Bounds-Wn} to find good bounds for the extreme zeros of $\mathcal{W}_N$, or equivalently, for the extreme zeros of $R_{N}$ on $\T$. With this purpose we formulate first the following lemma:
\begin{lemma}\label{lemma:inequality}
Let $a, b\in \R$, $0<q< 1$, and
\begin{equation}\label{defH}
h(x;a,b)=  \left(x- a\sqrt{1-x^2}\right)\left(x- b\sqrt{1-x^2}\right) , \quad x\in [-1,1].
\end{equation}
Then, the inequality $h(x;a,b)\geq q$ holds for 
\begin{equation}\label{solutionset}
x\in \left[-1,\frac{u^{(-)}}{\sqrt{1+(u^{(-)})^2}}\right] \cup \left[ \frac{u^{(+)}}{\sqrt{1+(u^{(+)})^2}}, 1\right],
\end{equation}
where $u^{(\pm)}=u^{(\pm)}(a,b,q)$ are the zeros of the polynomial \eqref{mainpolynomial}, as discussed in Section \ref{Sec-Intro}.
%(1-q) u^2 -(a+b)u + ab -1=0.

%where
%\begin{equation}\label{ypm}
%y_\pm := \frac{2(ab-q)}{a+b \mp \sqrt{D(q,a,b)}}, \quad D(q,a,b):= (a+b)^2 -4(1-q)(ab-q).
%\end{equation}
%
\end{lemma}

\begin{remark} \label{remarkq1}
	Recall that for $q\in [0,1)$, $u^{(\pm)}$ are real and finite, while for $q=1$ we assumed  $u^{(\pm)}(a,b,q)=\pm\infty$ when $\pm(a+b)\geq 0$. For instance, if $a+b>0$, then  
$$
u^{(-)} = \frac{ab-1}{a+b}, \quad u^{(+)}=+\infty.
$$
so the set \eqref{solutionset} is reduced to the single, leftmost interval.  Similar analysis is valid for $a+b<0$.
\end{remark}
\begin{proof}
Observe that $h(\pm 1; a,b)=1$, so the endpoints of $[-1,1]$ are always in the solution set of the inequality $h(x;a,b)\geq q$. Consider now $x\in (-1,1)$, and make in $h$ the change of variables $x=\cos (\theta/2)$, $\theta\in (0,2\pi)$:
$$
  %h(\cos \theta/2; a,b )=\sin^2 \left(\frac{\theta}{2}\right)  
    %\left(\cot \left(\frac{\theta}{2}\right)- a  \right)\left(\cot \left(\frac{\theta}{2}\right)- b  \right).
  h(\cos \theta/2; a,b )=\sin^2 \left(\theta/2\right)  \left(\cot \left(\theta/2\right)- a  \right)\left(\cot \left(\theta/2\right)- b  \right).
$$
Hence, using the basic identity $\sin^2 + \cos^2 =1$ we conclude that the inequality $h(\cos \theta/2; a,b )\geq q$ is equivalent to 
\[
\big(\cot \left(\theta/2\right)- a  \big)\big(\cot \left(\theta/2\right)- b  \big) \geq q\,\big(\cot^2 \left(\theta/2\right)+1  \big) .
\]
This is a quadratic inequality in $u=\cot (\theta/2)$: $(u-a)(u-b)\geq q(u^2+1)$, that is, $P(u)\geq 0$, where $P$ is the polynomial defined in \eqref{mainpolynomial}.
Notice that the discriminant $D(q,a,b) = (a+b)^2 - 4(1-q)(ab-q) $ is a quadratic polynomial in $q$ with the negative leading coefficient, and both $D(0,a,b)\geq 0$ and $D(1,a,b)\geq 0$, so that $D(q,a,b)\geq 0$ for all $0<q<1$. 

Since the leading coefficient of $P$ is positive, we conclude that $P(u)\geq 0$ when either $u\leq u^{(-)}$ or $u\geq u^{(+)}$, where $u^{(-)}\leq u^{(+)}$. It remains to use the monotonicity of $\cos (\theta/2)$ and the identity
\[
      x=\frac{u}{\sqrt{1+u^2}}
\]
to finish the proof.

\end{proof}

\begin{theo} \label{Thm-Estimates-for-Zero-Bounds-Wn}
   Let $N \geq 2$ and let $\{q_{n+1}\}_{n=1}^{N-1}$ be a finite scaling sequence for the chain sequence $\{d_{n+1}\}_{n=1}^{N-1}$, see Definition~\ref{def:scalingsequence}. Set  
\begin{equation}\label{def:AB}
 A_{N} = \min_{2 \leq n \leq N}   \frac{u_n^{(-)}}{\sqrt{1+\big(u_n^{(-)}\big)^2}} , \quad  
 B_{N} = \max_{2 \leq n \leq N}   \frac{u_n^{(+)}}{\sqrt{1+\big(u_n^{(+)}\big)^2}}  ,
\end{equation}
where $u_n^{(\pm)}$ were defined in \eqref{alternativeU}.

Then the zeros of $\mathcal{W}_N$ lie in $(A_{N},B_{N})$ and consequently, the zeros of $R_{N}$ lie within the open arc 
$ \mathcal{A}\big(2\arccos(B_{N}), 2\arccos(A_{N})\big) $.  
 
\end{theo} 
\begin{proof}
Given a finite positive chain sequence $\{\hat{d}_{n+1}\}_{n=1}^{N-1}$, the comparison test	\cite[Theorem 5.7]{Chihara-Book} assures that if 
\begin{equation} \label{comparisontest}
0 < \fd_{n+1}(x) \leq \hat{d}_{n+1} , \quad n =1,2, \ldots, N-1,
\end{equation}
then  $\{\fd_{n+1}(x)\}_{n=1}^{N-1}$ is also a finite positive chain sequence. From the definition  \eqref{definitionFD} it follows that \eqref{comparisontest} is equivalent to the inequality
\begin{equation} \label{Eq-Inequality-x}
h(x;c_n, c_{n+1})\geq q_{n+1},  \quad 1 \leq n \leq N-1,
\end{equation}
with $h$ defined in \eqref{defH}. Assume first that $q_{n+1}<1$.

Applying Lemma~\ref{lemma:inequality} we conclude that with the definitions given in the statement of the theorem, $\{\fd_{n+1}(A_N)\}_{n=1}^{N-1}$ and $\{\fd_{n+1}(B_N)\}_{n=1}^{N-1}$ are finite positive chain sequences, or in other words, condition \textit{b)} of Theorem \ref{Thm-General-Zero-Bounds-Wn} holds.

It is also easy to check that 
\[
   (u_{n+1}^{(+)} - c_{n})(c_{n} - u_{n+1}^{(-)}) = \frac{q_{n+1}}{1-q_{n+1}} (1+c_{n}^2)
\]
and
\[
   (u_{n+1}^{(+)} - c_{n+1})(c_{n+1} - u_{n+1}^{(-)}) = \frac{q_{n+1}}{1-q_{n+1}} (1+c_{n+1}^2),
\] 
for $n = 1, 2, \dots, N-1$. Since $u_{n+1}^{(+)}>u_{n+1}^{(-)}$, we conclude that $u_{n+1}^{(-)} < c_n$, $u_{n+1}^{(-)} < c_{n+1}$, $u_{n+1}^{(+)} > c_{n}$ and $u_{n+1}^{(+)} > c_{n+1}$ for  $1 \leq n \leq N-1$.
Direct calculation shows that   the values obtained for $A_N$ and $B_N$ also satisfy part \textit{a)} of Theorem \ref{Thm-General-Zero-Bounds-Wn}. 

Finally, we can use Remark~\ref{remarkq1} to conclude that the theorem also remains valid when $q_{n+1} = 1$.  
\end{proof} 

As it follows from the proof of Theorem~\ref{Thm-Estimates-for-Zero-Bounds-Wn}, the smaller the scaling sequence $\{q_{n+1}\}$ is, the tighter the bounds  $A_N$ and $B_N$ are. This is easier to observe when the measure is real-symmetric (that is, all its Verblunsky coefficients are real and all values  $c_n=0$): in this case
\[
    u_{n+1}^{(-)} = - \frac{\sqrt{q_{n+1}}}{\sqrt{1-q_{n+1}}} \quad \mbox{and} \quad u_{n+1}^{(+)} =  \frac{\sqrt{q_{n+1}}}{\sqrt{1-q_{n+1}}}.
\]
Combining it with Theorem~\ref{Thm-Zeros-in-Gaps} we see that for such measures the existence of a scaling   sequence $\{q_{n+1}\}_{n=1}^{\infty}$ such that  $q_{n+1} < q < 1$, $n \geq 1$, implies that the measure lives inside a symmetric arc and its support has a gap in a  neighborhood of $z=1$.  Examples of measures with such a property were given in \cite{Zhedanov-JAT98} by means of the Delsarte-Genin transformation and an additional scaling parameter.  

It is convenient to specify the statement of Theorem~\ref{Thm-Estimates-for-Zero-Bounds-Wn} for the case of the trivial scaling sequence  $q_{n+1} = 1$, $1 \leq n \leq N-1$. Recall that by the definition in Section~\ref{Sec-Intro}, 
\[
\pm(c_n+c_{n+1}) \geq 0 \quad \Rightarrow \quad u_{n+1}^{(\pm)}=u^{(\pm)}(c_n,c_{n+1},1)=\pm \infty.
\]
It means that if at least for one $n\in \{1, \dots, N-1\}$, $c_n+c_{n+1}\geq 0$, the value of $B_N$ in \eqref{def:AB} is $1$; analogously, if at least for one $n\in \{1, \dots, N-1\}$, $c_n+c_{n+1}\leq 0$,  $A_N=-1$. We can restate it as the following corollary: 

\begin{coro} \label{Coro-Estimates-for-Zero-Bounds-Wn}
In the conditions of Theorem~\ref{Thm-Estimates-for-Zero-Bounds-Wn} take  $q_{n+1} = 1$, $1 \leq n \leq N-1$. Then
\begin{itemize}
\item[(a)] If $c_n+c_{n+1}>0$ for all $n=1, \dots, N-1$, then  the zeros of $\mathcal{W}_N$ lie in $(A_{N},1)$ and the zeros of $R_{N}$ lie within the open arc $\mathcal{A}(0, 2\arccos(A_{N}))$, with $A_N$ defined in \eqref{def:AB}.
\item[(b)] If $c_n+c_{n+1}<0$ for all $n=1, \dots, N-1$, then  the zeros of $\mathcal{W}_N$ lie in $(-1, B_{N})$ and the zeros of $R_{N}$ lie within the open arc $\mathcal{A}( 2\arccos(B_{N}),2\pi)$, with $B_N$ defined in \eqref{def:AB}.
\end{itemize}
\end{coro} 
Notice that if $c_{n}+c_{n+1}$ changes sign, the statement of Corollary~\ref{Coro-Estimates-for-Zero-Bounds-Wn} becomes trivial.

We can formulate a somewhat weaker form of Theorem~\ref{Thm-Estimates-for-Zero-Bounds-Wn}  by observing that either 
\[
    x - c_n \sqrt{1-x^2} \leq -\sqrt{q_{n}^{(1,N)}}, \quad 1 \leq n \leq N,
\]
or 
\[
    x - c_n \sqrt{1-x^2} \geq \sqrt{q_{n}^{(1,N)}}, \quad 1 \leq n \leq N, 
\]
for $x \in [-1,1]$ are sufficient for \eqref{Eq-Inequality-x}. Here, $q_{1}^{(1,N)} = q_{2}$, $q_{n}^{(1,N)} = \max\{q_{n},q_{n+1}\}$, $2 \leq n \leq N-1$ and $q_{N}^{(1,N)} = q_{N}$. These inequalities are equivalent to 
\begin{equation}  \label{Eq-Inequalities-u1n}
   2 \cot^2(\theta/4) - 2c_n \cot(\theta/4) \leq \big(1 - \sqrt{q_{n}^{(1,N)}}\big)\, (1 + \cot^2(\theta/4)), \quad 1 \leq n \leq N,
\end{equation}
and 
\begin{equation} \label{Eq-Inequalities-v1n}
   2 \cot^2(\theta/4) - 2c_n \cot(\theta/4) \geq \big(1 + \sqrt{q_{n}^{(1,N)}}\big)\, (1 + \cot^2(\theta/4)), \quad 1 \leq n \leq N,
\end{equation}
with $x=\cos(\theta/2)$.  With similar manipulations as in the proof of Theorem \ref{Thm-Estimates-for-Zero-Bounds-Wn}, we get
\begin{theo} \label{Thm-Estimates2-for-Zero-Bounds-Wn}
For $N \geq 2$, let $\{q_{n+1}\}_{n=1}^{N-1}$ be a scaling sequence  for a finite positive chain sequence	$\{ d_{n+1}\}_{n=1}^{N-1}$, and 	$\{ c_{n}\}_{n=1}^{N-1}$ an arbitrary finite real sequence. Define 
\[
q_{1}^{(1,N)} = q_{2}, \quad q_{n}^{(1,N)} = \max\{q_{n},q_{n+1}\}, \quad 2 \leq n \leq N-1, \quad \text{and} \quad q_{N}^{(1,N)} = q_{N},
\]
as well as
\begin{equation}\label{def_un1}
\dsp u_{1,n} = \frac{c_{n}+ \sqrt{c_{n}^{2}+ (1-q_{n}^{(1,N)})}}{1+\sqrt{q_{n}^{(1,N)}}} \quad \mbox{and} \quad 
\dsp v_{1,n} = \frac{1+\sqrt{q_{n}^{(1,N)}}}{-c_{n}+ \sqrt{c_{n}^{2}+ (1-q_{n}^{(1,N)})}}.  
\end{equation}
If 
\begin{equation}\label{def:ABnew}
 A_{N} = \min_{1 \leq n \leq N}   \frac{u_{1,n}^2-1}{u_{1,n}^2+1} , \quad  
 B_{N} = \max_{1 \leq n \leq N}   \frac{v_{1,n}^2-1}{v_{1,n}^2+1}   ,
\end{equation}
then the zeros of $\mathcal{W}_N$, generated by \eqref{TTRR-for-Wn}, lie in $(A_{N},B_{N})$ and the zeros of $R_{N}$ lie within the open arc  
$$\mathcal{A}(2\arccos(B_{N}), 2\arccos(A_{N})).$$ 

\end{theo} 

\begin{remark} Considering the behavior of the function $x-c_n\sqrt{1-x^2}$, the solutions for $u_{1,n}$ in Theorem \ref{Thm-Estimates2-for-Zero-Bounds-Wn} are obtained from (\ref{Eq-Inequalities-u1n}) and the solutions for $v_{1,n}$ are obtained from (\ref{Eq-Inequalities-v1n}). 
\end{remark}

Finally, specifying Theorem~\ref{Thm-Estimates2-for-Zero-Bounds-Wn} for the trivial scaling sequence we obtain 
\begin{coro} \label{Coro-Estimates2-for-Zero-Bounds-Wn}
The assertion of Theorem~\ref{Thm-Estimates2-for-Zero-Bounds-Wn} is true if we take
\[
  \begin{array}l
    \dsp  u_{1,n} = \frac{c_{n}+ |c_{n}|}{2} \quad \mbox{and} \quad v_{1,n} = \frac{2}{-c_n + |c_n|}. 
  \end{array}
\]

\end{coro}

%%%%%%%%%%%%%%%%%%%%%%%%%%%%%%%%%%%%%%%%%%%%%%%%%%%%%%%%%%%%%%%
\setcounter{equation}{0}
\section{Bounds for the support of the measure}  \label{Sec-BSupportM}

Now by making $n\to \infty$ and  considering appropriate infinite positive chain sequence $\{\hat{d}_{n+1}\}_{n=0}^{\infty}$ we can combine Theorem \ref{Thm-Estimates-for-Zero-Bounds-Wn} (for bounds for the zeros of $\mathcal{W}_n$ and $R_n$) and Theorem \ref{ThmA-Support} in order to establish some results about the support of the measure $\mu(0;.)$ from its $(c_n, d_n)$-parametrization. 

\begin{theo} \label{Thm-Estimate-1-for-Support} Given the real sequence $\{c_n\}_{n=1}^{\infty}$ and the non SP positive chain sequence  $\{d_{n+1}\}_{n=1}^{\infty}$, let $\mu(t;.)$, $0 \leq t < 1$,  be the corresponding family of positive measures  \eqref{Eq-Family-of-Measures}. Let also $\{q_{n+1}\}_{n=1}^{\infty}$ be a scaling sequence for $\{d_{n+1}\}_{n=1}^{\infty}$, see Definition~\ref{def:scalingsequence}. Set $A_N$ and $B_N$ as in \eqref{def:AB}, that is, 
\[
A_{N} = \min_{2 \leq n \leq N}   \frac{u_n^{(-)}}{\sqrt{1+\big(u_n^{(-)}\big)^2}} , \quad  
B_{N} = \max_{2 \leq n \leq N}   \frac{u_n^{(+)}}{\sqrt{1+\big(u_n^{(+)}\big)^2}}  ,
\]
where $u_n^{(\pm)}$ were defined in \eqref{alternativeU}.	Then
\[   
\supp (\mu(0;\cdot)) \subseteq \mathcal{A}[\vartheta_1, \vartheta_2] ,
\]
where
$$
\vartheta_1=2 \lim_{N \to \infty}   \arccos(B_N), \quad \vartheta_2=2 \lim_{N \to \infty}   \arccos(A_N).
$$
\end{theo} 
Clearly, in Theorem \ref{Thm-Estimate-1-for-Support},   $\{B_N\}$ (resp., $\{A_N\}$) is an increasing (resp., decreasing) sequence, thus the limits exist.

Theorem \ref{Thm-Restriction-1-for-cn-dn2} now is a direct corollary of Theorem~\ref{Thm-Estimate-1-for-Support}. Indeed,  if $\{c_n\}$ and $\{d_{n+1}\}$ are such that 
\[
u_n \geq  \cot(\vartheta_2/2)\quad \mbox{and} \quad v_n \leq \cot(\vartheta_1/2), \quad n \geq 2,
\]
then by Theorem \ref{Thm-Estimate-1-for-Support} the zeros of $R_n$, $n\geq 1$, stay within $\mathcal{A}(\vartheta_1, \vartheta_2)$.  Hence, by Theorem \ref{ThmA-Support} the support of the associated measure $\mu(0;.)$ also belongs to the closed arc $\mathcal{A}[\vartheta_1, \vartheta_2]$. This concludes the proof. 

Similarly, with the use of Theorem \ref{Thm-Estimates2-for-Zero-Bounds-Wn} we have the following. 

\begin{theo} \label{Thm-Estimate-2-for-Support}
Given the real sequence $\{c_n\}_{n=1}^{\infty}$ and the non SP positive chain sequence  $\{d_{n+1}\}_{n=1}^{\infty}$,  let $\mu(t;.)$, $0 \leq t < 1$,  be the corresponding family of positive measures  \eqref{Eq-Family-of-Measures}.   
If $\{q_{n+1}\}_{n=1}^{\infty}$ is a scaling sequence  for	$\{ d_{n+1}\}_{n=1}^{\infty}$, define 
\[
q_{1}^{(1)} = q_{2}, \quad q_{n}^{(1)} = \max\{q_{n},q_{n+1}\}, \quad  n \geq 2, 
\]
as well as
\begin{equation*}
\dsp u_{1,n} = \frac{c_{n}+ \sqrt{c_{n}^{2}+ (1-q_{n}^{(1)})}}{1+\sqrt{q_{n}^{(1)}}} \quad \mbox{and} \quad 
\dsp v_{1,n} = \frac{1+\sqrt{q_{n}^{(1)}}}{-c_{n}+ \sqrt{c_{n}^{2}+ (1-q_{n}^{(1)})}}.  
\end{equation*}
Then
\[   
	   \supp (\mu(0;\cdot)) \subseteq \mathcal{A}[\vartheta_1, \vartheta_2] ,
\]
where
\[
	\cot(\vartheta_1/4)= \sup_{n\geq 1}   v_{1,n}   \quad \mbox{and} \quad \cot(\vartheta_2/4)= \inf_{n\geq 1}   u_{1,n}.
\]

\end{theo}

The following corollary to Theorem \ref{Thm-Restriction-1-for-cn-dn2} is a consequence of Corollary \ref{Coro-Estimates-for-Zero-Bounds-Wn}. 

\begin{coro} \label{Coro-Restriction-1-for-cn-dn}
   Given the set of complex numbers  $\{\alpha_n\}_{n=0}^{\infty}$,  where $|\alpha_n| < 1$, $n\geq 0$, let $\mu$ be the nontrivial positive measure on the unit circle for which  $\{\alpha_n\}_{n=0}^{\infty}$ are the associated Verblunsky coefficients. Let the real sequence  $\{c_n\}$ and the positive chain sequence $\{d_{n+1}\}$ be  given by \eqref{cSeq}--\eqref{parametricSeq}. \\[-1ex] 
   
\noindent $a)$\  Let $c_n + c_{n+1} > 0$, $n \geq 1$. For a given $\vartheta_2$ such that $0 < \vartheta_2 \leq 2\pi$, \  if 
\[
    \cot(\vartheta_2/2) \leq \frac{c_{n}c_{n+1}-1}{c_{n} + c_{n+1}}, \quad n \geq 1,
\]   
then the support of $\mu$ lies within the closed arc $\mathcal{A}[0, \vartheta_2]$. \\[-1ex]

\noindent $b)$\  Let $c_n + c_{n+1} < 0$, $n \geq 1$. For a given $\vartheta_1$ such that $0 \leq \vartheta_1 < 2\pi$, \  if   
\[
    \frac{c_{n}c_{n+1}-1}{c_{n} + c_{n+1}} \leq \cot(\vartheta_1/2), \quad n \geq 1,
\]   
then the support of $\mu$ lies within the closed arc $\mathcal{A}[\vartheta_1, 2\pi]$.
\end{coro}

Using Theorem \ref{Thm-Estimate-2-for-Support} we can also state the following.

\begin{theo} \label{Thm-Restriction-2-for-cn-dn}
Given a nontrivial positive measure on $\T$ whose Verblunsky coefficients are $\{\alpha_n\}_{n=0}^{\infty}$,  with $|\alpha_n| < 1$, $n\geq 0$, and whose $(c_n,d_n)$-parametrization is  given by \eqref{cSeq}--\eqref{parametricSeq}.
If $\{q_{n+1}\}_{n=1}^{\infty}$ is a scaling sequence  for	$\{ d_{n+1}\}_{n=1}^{N-1}$, define 
$$
q_{1}^{(1)} = q_{2}, \quad q_{n}^{(1)} = \max\{q_{n},q_{n+1}\}, \quad  n \geq 2. 
$$ 
 For $0 \leq \vartheta_1 < \vartheta_2 \leq 2\pi$,  if 
\[
  \begin{array}l
    \dsp  \cot(\vartheta_2/4) \leq \frac{c_{n}+ \sqrt{c_{n}^{2}+ (1-q_{n}^{(1)})}}{1+\sqrt{q_{n}^{(1)}}} \quad \mbox{and} \quad %\\[3ex]
    \dsp  \frac{1+\sqrt{q_{n}^{(1)}}}{-c_{n}+ \sqrt{c_{n}^{2}+ (1-q_{n}^{(1)})}} \leq \cot(\vartheta_1/4),  
  \end{array}
\]
for $n \geq 1$, then the support of $\mu(0;.)$ stays within the closed arc $\mathcal{A}[\vartheta_1, \vartheta_2]$.  

\end{theo}

We finish this section providing a proof of Theorem \ref{Thm-Restriction-3-for-cn-dn} stated in the introduction.

\begin{proof}[Proof of Theorem \ref{Thm-Restriction-3-for-cn-dn}.]
  Let us consider the measure $\tilde{\mu}$ obtained from $\mu$ by a rotation of angle $2\pi-\vartheta_2$.  That is,  $\tilde{\mu}(z) = \mu(z e^{-i(2\pi-\vartheta_2)})$. Hence, if $\mu$ has a gap in $\mathcal{A}(\vartheta_1, \vartheta_2)$ then the support of the measure $\tilde{\mu}$ is in $\mathcal{A}[0, 2\pi+\vartheta_1-\vartheta_2]$. 

It is known that the Verblunsky coefficients of  $\tilde{\mu}$ in terms of the Verblunsky coefficients $\mu$ are given by
\[
     \tilde{\alpha}_{n} = e^{-i(n+1)(2\pi-\vartheta_2)} \alpha_{n} = e^{i(n+1)\vartheta_2}  \alpha_{n}, \quad n \geq 0. 
\]
Hence, if $\tilde{\Phi}_n$ are the OPUC with respect to the measure $\tilde{\mu}$ then by Theorem \ref{Thm-Zeros-in-Gaps} (part b) the zeros of the polynomials 
\[
   \tilde{R}_n(z) = \frac{\prod_{j=1}^{n} \big[1- \tilde{\tau}_{j-1}\tilde{\alpha}_{j-1}\big]}{\prod_{j=1}^{n} \big[1-\Re(\tilde{\tau}_{j-1}\tilde{\alpha}_{j-1})\big]} \frac{z\tilde{\Phi}_{n}(z) - \tilde{\tau}_{n} \tilde{\Phi}_{n}^{\ast}(z)}{z-1}, \quad n\geq 1, 
\]
lie within the arc $\mathcal{A}(0, 2\pi+\vartheta_1-\vartheta_2)$. From the results given in Section \ref{sec:CDandPositiveCS}, $\tilde{\tau}_n$ above are such that 
\begin{equation} \label{Eq-RF-tildeTau}
      \tilde{\tau}_{0} = \frac{\tilde{\Phi}_0(1)}{\tilde{\Phi}_{0}^{\ast}(1)} = 1 \quad \mbox{and} \quad 
      \tilde{\tau}_{n} = \frac{\tilde{\Phi}_{n}(1)}{\tilde{\Phi}_{n}^{\ast}(1)} = \tilde{\tau}_{n-1} \frac{1 - \overline{\tilde{\tau}_{n-1}\tilde{\alpha}_{n-1}}}{1 - \tilde{\tau}_{n-1}\tilde{\alpha}_{n-1}}, \ \ n \geq 1.
\end{equation}
Moreover, the sequence $\{\tilde{R}_n\}$ satisfies the three term recurrence formula 
\[
   \tilde{R}_{n+1}(z) = \left[(1+ic_{n+1}^{(\vartheta_2)})z + (1-ic_{n+1}^{(\vartheta_2)})\right] \tilde{R}_{n}(z) - 4d_{n+1}^{(\vartheta_2)} z \tilde{R}_{n-1}(z),
\]
with $\tilde{R}_{0}(z) = 1$ and $\tilde{R}_{1}(z) = (1+ic_{1}^{(\vartheta_2)})z + (1-ic_{1}^{(\vartheta_2)})$, where 
\begin{equation*}\label{Eq-CoeffsTTRR-2}
  \begin{array}{l}
    \dsp c_{n}^{(\vartheta_2)} = \frac{-\Im (\tilde{\tau}_{n-1}\tilde{\alpha}_{n-1})} {1-\Re(\tilde{\tau}_{n-1}\tilde{\alpha}_{n-1})}  \quad  \mbox{and} \quad  d_{n+1}^{(\vartheta_2)} = (1-g_{n})g_{n+1},
  \end{array} 
\end{equation*}
for $n \geq 1$, with  
\[
       g_{n} = \frac{1}{2} \frac{\big|1 - \tilde{\tau}_{n-1} \tilde{\alpha}_{n-1}\big|^2}{\big[1 - \Re(\tilde{\tau}_{n-1}\tilde{\alpha}_{n-1})\big]}, \quad n \geq 1. 
\]
With $\tau_{n-1}^{(\vartheta_2)} = e^{i n \vartheta_2} \tilde{\tau}_{n-1}$, $n \geq 1$, we easily observe that $\tilde{\tau}_{n-1} \tilde{\alpha}_{n-1} = \tau_{n-1}^{(\vartheta_2)} \alpha_{n-1}$, $n \geq 1$. Hence, the coefficients $c_n^{(\vartheta_2)}$ and $d_{n+1}^{(\vartheta_2)}$ can be given as in the theorem, and the recurrence formula (\ref{Eq-RF-tildeTau}) for $\tilde{\tau}_n$ leads to the recurrence formula for $\tau_n^{(\vartheta_2)}$ given in the theorem. 

Hence,  the theorem follows from Theorem \ref{Thm-General-Zero-Bounds-Wn}. \end{proof}

%%%%%%%%%%%%%%%%%%%%%%%%%%%%%%%%%%%%%%%%%%%%%%%%%%%%%%%%%%%%%%%
\setcounter{equation}{0}
\section{Further results about scaling sequences} \label{Sec-SORPCS}

As it follows from our previous considerations, scaling sequences $\{q_{n+1}\}$ play a fundamental role in the expression for the bounds for the extreme zeros. Recall that $q_{n+1}\equiv 1$ is always a scaling sequence. A natural question is which other constant sequences are scaling sequences for the given positive chain sequence. This can be answered in terms of the zeros of the symmetric polynomials $\mathcal{W}_n$ generated by the recurrence relation~\eqref{TTRR-for-Wn} with $c_n\equiv 0$.

Namely, a direct consequence of Theorem \ref{Thm-General-Zero-Bounds-Wn} is the following result:
\begin{lemma} \label{Lemma-Dominant-FPCS}
Let $N \geq 2$ and let $\{d_{n+1}\}_{n=1}^{N-1}$ be a finite positive chain sequence. Then the constant sequence $\{q_{n+1}=q\}_{n=1}^{N-1}$ is a finite scaling sequence for $\{d_{n+1}\}_{n=1}^{N-1}$ if and only if 
$$
q> x_{N,1}^2.
$$ 
Here, $x_{N,1}$ is the largest zero of the symmetric polynomial $\mathcal{W}_N$ obtained from the three term recurrence formula 
\[
    \mathcal{W}_{n+1}(x) = x\,\mathcal{W}_n(x) - d_{n+1} \mathcal{W}_{n-1}(x), \quad n = 1,2, \ldots , N-1,
\]
with 
$\mathcal{W}_0(x) = 1$ and $\mathcal{W}_1(x) = x$. 
\end{lemma}

Applying this Lemma with  $\{d_{n+1}\}=\{1/4\}$ (That is, $\mathcal{W}_{n}$ are the Chebyshev polynomial of the second kind), we can conclude that the largest finite constant positive chain sequence of $N-1$ elements is
\[
   \{d^{(N-1)}\}_{n=1}^{N-1} = \Big\{\frac{1}{4\cos^{2}(\pi/(N+1))}\Big\}_{n=1}^{N-1};
\]
this result was already established by Ismail and Li \cite{IsmailLi-1992}. 

We can extend this statement to infinite positive chain sequences $\{d_{n+1}\}_{n=1}^{\infty}$:
\begin{lemma} \label{Lemma-Dominant-PCS}
Let  $\{d_{n+1}\}_{n=1}^{\infty}$ be a positive chain sequence. Then the constant sequence $\{d_{n+1}=q\}_{n=1}^{\infty}$ is a finite scaling sequence for $\{d_{n+1}\}_{n=1}^{\infty}$ if and only if 
\[
      q\geq \xi^2=\lim_{n\to \infty} x_{n,1}^2,
\]
where  $x_{n,1}$ is the largest zero of the symmetric polynomial $\mathcal{W}_n$ given by the three term recurrence formula 
\[
    \mathcal{W}_{n+1}(x) = x\,\mathcal{W}_n(x) - d_{n+1} \mathcal{W}_{n-1}(x), \quad n\geq 1,
\]
with 
$\mathcal{W}_0(x) = 1$ and $\mathcal{W}_1(x) = x$.
\end{lemma}

We finish this section with an example of a positive chain sequence that we will use in Section \ref{Sec-Examples}. For $\lambda \geq -1/2$, the sequence % $\{d_{n+1}^{(\lambda)}\}_{n=1}^{\infty}$  is  
\begin{equation}\label{def:Gegenbauer}
   d_{1,n}^{(\lambda)}  =  d_{n+1}^{(\lambda)} = \frac{1}{4} \frac{n(n+2\lambda+1)}{(n+\lambda)(n+\lambda+1)}, \quad n \geq 1,
   \end{equation}  
is a positive chain sequence. %  (although we are interested in the case ). 
We have
\[
    d_{1,n}^{(\lambda)}= (1-\mathfrak{m}_{n-1})\mathfrak{m}_{n} =  (1-M_{1,n-1})M_{1,n} , \quad n \geq 1,
\]  
where
\[
     \mathfrak{m}_{n} = \frac{n}{2(n+\lambda+1)} \quad \mbox{and} \quad \quad M_{1,n} = \frac{n+2\lambda+1}{2(n+\lambda+1)}, \quad n \geq 0,
\]
is the minimal and maximal parameter sequences of $\{d_{1,n}^{(\lambda)}\}_{n=1}^{\infty}$, respectively (see \cite{Costa-Felix-Ranga-JAT2013}). Observe that for $\lambda = -1/2$  these parameter sequences coincide, so $\{d_{1,n}^{(-1/2)}\}_{n=1}^{\infty}$ is a single parameter positive chain sequence. 

Polynomials $\{\mathcal{W}_n\}_{n=0}^{\infty}$ generated by $\mathcal{W}_0(x) = 1$, $ \mathcal{W}_1(x) = x$ and
$$
    \mathcal{W}_{n+1}(x) = x\,\mathcal{W}_n(x) - d_{n+1}^{(\lambda)}  \mathcal{W}_{n-1}(x), \quad n\geq 1,
$$
coincide with the monic ultraspherical (or Gegenbauer) polynomials $\{C_n^{(\lambda+1)}\}$, which for  $\lambda= -1/2$ are the monic Legendre polynomials. 

We also have
\[  
   \dsp\frac{1}{4} - d_{n+1}^{(\lambda)} 
     = \frac{1}{4} \frac{\lambda(\lambda+1)}{(n+\lambda)(n+\lambda+1)}, 
          \quad n\geq 1,
\]
and
\[ 
   \begin{array}{ll} 
     \dsp d_{n+1}^{(-1/2)} - d_{n+1}^{(\lambda)}  
     &= \dsp \frac{1}{4}  \frac{1/4}{(n^2-1/4)} 
        + \frac{1}{4} \frac{\lambda(\lambda+1)}{(n+\lambda)(n+\lambda+1)}\\[3ex]
     &= \dsp \frac{n}{4}\frac{n(\lambda+1/2)^2 + (\lambda+1/2)/2} 
                         {(n^2-1/4)(n+\lambda)(n+\lambda+1)}, \quad n\geq 1.
   \end{array} 
\]
Thus,  we can state the following:
\begin{lemma}  \label{Lemma-Dominant-lambda-FPCS}
Let $\lambda > -1/2$ and $N \geq 2$. Then the sequence 
$\dsp \{[\cos(\pi/(2N))]^{-2} d_{n+1}^{(-1/2)}\}_{n=1}^{N-1}$ is a finite positive chain sequence such that 
\[ 
    d_{n+1}^{(\lambda)} < d_{n+1}^{(-1/2)} 
      < \frac{1}{\cos^2(\pi/(2N))}d_{n+1}^{(-1/2)} 
      < \frac{1}{(x_1^{(N)})^2} d_{n+1}^{(-1/2)}, \quad n = 1, 2, \ldots, N-1.  
\]
Here, $x_1^{(N)}$ is the largest zero of the $N^{th}$ degree Legendre polynomial. 
\end{lemma}

\begin{proof}
By Lemma \ref{Lemma-Dominant-FPCS}, $q$ is the constant finite parameter sequence for $\dsp \{  d_{n+1}^{(-1/2)}\}_{n=1}^{N-1}$ if  $q> [x_1^{(N)}]^{2}$. It remains to use  that  $x_1^{(N)} < \cos(\pi/(2N))$, see \cite[Thm.\,6.21.3]{Szego-Book}.
\end{proof}

%%%%%%%%%%%%%%%%%%%%%%%%%%%%%%%%%%%%%%%%%%%%%%%%%%%%%%%%%%%%%%%
\setcounter{equation}{0}
\section{Examples}  \label{Sec-Examples}

\noindent {\bf Example 1 }(Geronimus polynomials). Let $\alpha \in \D$, and consider a measure $\mu^{(\alpha)}$ on $\T$ with constant Verblunsky coefficients,
\[
\alpha_{n} =-\overline{\Phi_{n+1}^{(\alpha)}(0)}  =  \alpha , \quad n \geq 0.
\]
This case was studied by Geronimus in \cite{Geronimus-AMSTransl-1977}; it is known  (see also \cite{GolinNevaiAssche-1995}, \cite{GolinNevaiPinterAssche-1999}  and \cite[p.\,83]{Simon-Book-p1}) that for $\Re(\alpha) + |\alpha|^2 \leq 0$, $\supp(\mu^{(\alpha)})= \mathcal{A}[\theta_{|\alpha|}, 2\pi- \theta_{|\alpha|}] $, while for $\Re(\alpha) + |\alpha|^2 > 0$, 
$\supp(\mu^{(\alpha)})= \mathcal{A}[\theta_{|\alpha|}, 2\pi- \theta_{|\alpha|}] \cup \{w_\alpha\}$, with
\begin{equation} \label{Eq-Infor-Measure}
\theta_{|\alpha|} = 2\arcsin(|\alpha|) \quad  \text{and} \quad %e^{i\vartheta_{\alpha}} =
 w_{\alpha} = \frac{1+\overline{\alpha}}{1+\alpha} .
\end{equation}
Additionally, 
\[
d \mu^{(\alpha)}(e^{i\theta})=  \frac{\sqrt{\cos^2(\theta_{|\alpha|}/2)-\cos^2(\theta/2)}}{2 \pi |1+\alpha|\,\sin((\theta-\vartheta_{\alpha})/2)} d\theta+ t_\alpha \delta_{w_\alpha}, 
\]
where 
\[
   t_{\alpha} =\max \left\{ \frac{2(\Re(\alpha) + |\alpha|^2)}{|1+\alpha|^2},0 \right\}.
\]

Let $\Im(\alpha) \neq 0$, and denote $\vartheta_{\alpha}:=\arg (w_\alpha) \in (-\pi , \pi)$. Consider the measure $\mu$ obtained by rotating $\mu^{(\alpha)}$  by the angle  $-\vartheta_{\alpha}$:  
\[
    \mu(z) = \mu^{(\alpha)}(w_{\alpha} z).
\]
It is well known that the Verblunsky coefficients associated with $\mu$ are $\alpha_{n} = w_{\alpha}^{n+1} \alpha$, $n \geq 0$.

The measure $\mu$ is such that its support has gaps in $(0, \theta_{|\alpha|}-\vartheta_{\alpha})$ and $(2\pi-\theta_{|\alpha|}-\vartheta_{\alpha}, 2\pi)$. In other words, the support of  $\mu(0;.)$, see (\ref{Eq-Family-of-Measures}), is the arc $\mathcal{A}[\theta_{|\alpha|}-\vartheta_{\alpha}, 2\pi-\theta_{|\alpha|}-\vartheta_{\alpha}]$.
Since $\sin(\theta_{|\alpha|}/2) = |\alpha|$ and $\sin(\vartheta_{\alpha}/2) = (1+\overline{\alpha})/|1+\alpha|$, we easily obtain that
\[
  \begin{array}{l}
   \cot\left((\theta_{|\alpha|}-\vartheta_{\alpha})/2\right) 
      \dsp = \frac{\sqrt{1-|\alpha|^2}\,\big(1 + \Re(\alpha)\big) - |\alpha|\, \Im(\alpha)} {|\alpha|\,\big(1 + \Re(\alpha)\big) + \sqrt{1-|\alpha|^2}|\, \Im(\alpha)} \\[3ex]
  \quad \dsp = \frac{\big[\sqrt{1-|\alpha|^2}\,\big(1 + \Re(\alpha)\big) - |\alpha|\, \Im(\alpha)\big]\big[|\alpha|\,\big(1 + \Re(\alpha)\big) + \sqrt{1-|\alpha|^2}|\, \Im(\alpha) \big]} {|\alpha|^2\,\big(1 + \Re(\alpha)\big)^2 - (1-|\alpha|^2)|\, \big(\Im(\alpha)\big)^2}
  \end{array}
\]
and
\[
  \begin{array}l 
   \cot\left((2\pi-\theta_{|\alpha|}-\vartheta_{\alpha})/2\right) 
     \dsp = \frac{\sqrt{1-|\alpha|^2}\,\big(1 + \Re(\alpha)\big) + |\alpha|\, \Im(\alpha)} {-|\alpha|\,\big(1 + \Re(\alpha)\big) + \sqrt{1-|\alpha|^2}|\, \Im(\alpha)}  \\[3ex]
   \quad \dsp = \frac{\big[\sqrt{1-|\alpha|^2}\,\big(1 + \Re(\alpha)\big) + |\alpha|\, \Im(\alpha)\big] \big[|\alpha|\,\big(1 + \Re(\alpha)\big) + \sqrt{1-|\alpha|^2}|\, \Im(\alpha) \big]} {-|\alpha|^2\,\big(1 + \Re(\alpha)\big)^2 + (1-|\alpha|^2)|\, \big(\Im(\alpha)\big)^2} .
  \end{array}
\]
With the observation $\big(1 + \Re(\alpha)\big)^2 + \big(\Im(\alpha)\big)^2 = |1+\alpha|^2$, we can write 
\begin{equation}\label{example11}
  \begin{array}{l}
   \cot\left((\theta_{|\alpha|}-\vartheta_{\alpha})/2\right) 
      \dsp = \frac{ -\Im(\alpha)\big(1 + \Re(\alpha)\big) + |\alpha|\sqrt{1-|\alpha|^2}\,|1+\alpha|^2} {\big(1 + \Re(\alpha)\big)^2 - (1-|\alpha|^2)|\, |1 + \alpha|^2}
  \end{array}
\end{equation}
and 
\begin{equation}\label{example12}
  \begin{array}{l}
   \cot\left((2\pi-\theta_{|\alpha|}-\vartheta_{\alpha})/2\right) 
      \dsp = \frac{ -\Im (\alpha)\big(1 + \Re (\alpha)\big) - |\alpha|\sqrt{1-|\alpha|^2}\,|1+\alpha|^2} {\big(1 + \Re (\alpha)\big)^2 - (1-|\alpha|^2)|\, |1 + \alpha|^2}. 
  \end{array}
\end{equation}

In order to apply Theorem \ref{Thm-Restriction-1-for-cn-dn2}  we need to calculate the values of   $\{c_n\}_{n=1}^{\infty}$ and $\{d_{n+1}\}_{n=1}^{\infty} = \{(1-g_n)g_{n+1}\}_{n=1}^{\infty}$ associated with the measure $\mu(z)= \mu^{(\alpha)}(w_{\alpha}z)$ using (\ref{cSeq}) and (\ref{parametricSeq}). From the recurrence relation for $\tau_n$ given by (\ref{recurrenceC}), we easily verify that $\tau_n = w_{\alpha}^{-n}$, $n \geq 0$. From this, 
\[
    c_n = \frac{- \Im (\alpha)}{1 + \Re (\alpha)}, \quad n \geq 1,
\]
and
\[
     d_{n+1} = \Big(1 - \frac{1-|\alpha|^2}{2 \big[1 + \Re (\alpha)\big]}\Big) \,\frac{1-|\alpha|^2}{2 \big[1 + \Re (\alpha)\big]} = \frac{(1-|\alpha|)^2|1+\alpha|^2 }{4 \big[1 + \Re (\alpha)\big]^2}, \quad n \geq 1. 
\]
Both the chain sequence $\{d_{n+1}\}_{n=1}^{\infty}$ and the parameter sequence $\{g_{n+1}\}_{n=0}^{\infty}$, with 
\[
    g_n = \frac{1-|\alpha|^2}{2 \big[1 + \Re (\alpha)\big]}, \quad n \geq 1,
\]
are constant  sequences. Thus, from Wall's criteria (see \cite[pg.\,101]{Chihara-Book}) $\{g_{n+1}\}_{n=0}^{\infty}$ is not a maximal parameter sequence if and only if $g_n < 1/2$, $n \geq 1$, or equivalently, when $\Re (\alpha) + |\alpha|^2 >  0$, which agrees with the criterion of existence of a pure point (at $z=1$) in  the rotated measure $\mu$. 

Observe that the constant chain sequence $\{d_{n+1}\}_{n=1}^{\infty}$ is such that $d_{n+1} \leq 1/4$, $n \geq 1$. Hence, 
\[
q_{n+1} = \frac{(1-|\alpha|)^2|1+\alpha|^2 }{\big(1 + \Re (\alpha)\big)^2},  \quad n \geq 1,
\]
is a constant scaling sequence for $\{d_{n+1}\}_{n=1}^{\infty}$. Since for $n \geq 1$,
\begin{equation*}
  \begin{array}{rl}
 c_{n}c_{n+1}+1 & \dsp = \frac{|1+\alpha|^2}{\big(1+\Re (\alpha)\big)^2},   \\[3ex]
 \sqrt{(c_{n}+c_{n+1})^2 - 4 (1-q_{n+1})(c_{n}c_{n+1} - q_{n+1})} 
 &\dsp = \sqrt{4q_{n+1}(c_{n}c_{n+1}+1-q_{n+1})} \\[1ex] 
 &\dsp= (1-|\alpha|^2)|\alpha|^2 \frac{|1+\alpha|^4}{\big(1+\Re (\alpha)\big)^4},
  \end{array}
\end{equation*}
Using the definition \eqref{alternativeU} and formulas \eqref{example11}--\eqref{example12} we obtain that
\[
  \begin{array}{ll}
    u_{n+1}^{(-)}=&\!\!u^{(-)}(c_n,c_{n+1},q_{n+1})= \cot\left((2\pi-\theta_{|\alpha|}-\vartheta_{\alpha})/2\right) \\[1ex] 
    u_{n+1}^{(+)}=&\!\!u^{(+)}(c_n,c_{n+1},q_{n+1})= \cot\left((\theta_{|\alpha|}-\vartheta_{\alpha})/2\right) . 
\end{array}
\]
In other words, the bounds given by Theorem \ref{Thm-Restriction-1-for-cn-dn2} are sharp.

\medskip

\noindent {\bf Example 2.}  Let $\mu^{(b_1, b_2, c)}$ be the probability measure on the unit circle for which the associated Verblunsky coefficients are given by
\[
   \alpha_{2n} = \frac{b_1 + i c}{1 + i c}, \quad  \alpha_{2n+1} = \frac{b_2 - i c}{1 + i c}, \quad n \geq 0.
\]
Here,  $b_1$, $b_2$ and  $c$ are real, and  $-1 < b_1 < 1$ and $-1 < b_2 < 1$.  

We can apply the results from \cite[Chapter 11]{Simon-Book-p2} in order to assert that the absolutely continuous part of the measure $\mu^{(b_1, b_2, c)}$ is in  $\mathcal{A}[\vartheta_1^{+}, \vartheta_1^{-}] \cup \mathcal{A}[\vartheta_2^{-}, \vartheta_2^{+}]$, where $0 \leq \vartheta_1^{+} < \vartheta_1^{-} \leq \pi$, $2\pi - \vartheta_2^{+} =  \vartheta_1^{+} = \vartheta^{+}(b_1,b_2, c)$ and $2\pi - \vartheta_2^{-} =  \vartheta_1^{-} = \vartheta^{-}(b_1,b_2, c)$, with 
\[
   \vartheta^{+}(b_1,b_2, c)  =\arccos\Big(\frac{c^2 - b_1b_2 + (1-b_1^2)^{1/2}(1-b_2^2)^{1/2}}{c^2 +1}\Big)  
\]
and
\[
   \vartheta^{-}(b_1,b_2, c) = \arccos\Big(\frac{c^2 - b_1b_2 - (1-b_1^2)^{1/2}(1-b_2^2)^{1/2}}{c^2 +1}\Big) .
\]
Moreover, the absolutely continuous part of $\mu^{(b_1, b_2, c)}$  is 
\[
      d\mu_{ac}^{(b_1,b_2,c)}(e^{i\theta}) = \const  \frac{\sqrt{(1-b_1^2)(1-b_2^2)-[c^2 - b_1b_2 - (c^2+1)\cos\theta]^2}}{|\sin\theta+c(1-\cos\theta)|}d\theta.
\]
From \cite{Simon-Book-p2} we also have that if $b_2+b_1 > 0$ then $\mu^{(b_1, b_2, c)}$ has a pure point at $z=1$; if $b_2 - b_1 > 0$ then $\mu^{(b_1, b_2, c)}$ has a pure point at $e^{i\vartheta(-c)}$, where
\[
    e^{i\vartheta(c)} =\frac{c^2-1}{c^2+1} + i \frac{2c}{c^2+1}. 
\] 

Now we consider the polynomials $R_n$ given by (\ref{Eq-CDKernel-POPUC}). It follows from parts $a)$ and $b)$ of Theorem \ref{Thm-Zeros-in-Gaps}  that the polynomial $R_n$ does not have any zeros outside the arc $\mathcal{A}[\vartheta_1^{+}, \vartheta_2^{+}]$.  From part $c)$ of Theorem \ref{Thm-Zeros-in-Gaps} we conclude that there is at most one zero of $R_n$ in the arc $\mathcal{A}[\vartheta_1^{-}, \vartheta_2^{-}]$ if $b_2-b_1 \leq 0$, and at most two zeros $R_n$ in $\mathcal{A}[\vartheta_1^{-}, \vartheta_2^{-}]$  if $b_2-b_1 > 0$.

Let us look now at the $(c_n,d_n)$-parametrization of $\mu^{(b_1,b_2,c)}$, and at the corresponding bounds for the zeros of $R_n$.  

From (\ref{recurrenceC}), $\tau_{2n} = 1$ and $\tau_{2n+1} = \frac{1+ic}{1-ic}$, $n \geq 0$.  Hence,  by (\ref{cSeq}) and (\ref{parametricSeq}), %in the three term recurrence formula (\ref{Eq-TTRR-Rn}) for $R_n$, 
\[
   c_n = (-1)^n c, \quad d_{n+1} = (1-g_n)g_{n+1}, \quad n \geq 1,
\] 
where
\[
    g_{2n-1} = \frac{1}{2}(1-b_1), \quad g_{2n}= \frac{1}{2}(1-b_2), \quad n \geq 1.
\]

From Corollary \ref{Coro-General2-Zero-Bounds-Wn} we can say that the  zeros of $R_n$ are outside  the arc \linebreak $\mathcal{A}\big(\vartheta(|c|), 2\pi-\vartheta(|c|)\big)$.  This is in some sense a sharp result because one of the extreme points of $\mathcal{A}\big(\vartheta(|c|), 2\pi-\vartheta(|c|)\big)$ coincide with the position of the (possible) pure point of $\mu^{(b_1, b_2, c)}$ at $e^{i\vartheta(-c)}$.  Observe that when $b_1=b_2=b$ then $e^{i\vartheta(-c)}$ also coincide with  $e^{i\vartheta^{-}(b,b, c)}$ or $e^{i[2\pi - \vartheta^{-}(b,b, c)]}$. 

From Theorem  \ref{Thm-Restriction-1-for-cn-dn2} (see also Theorem \ref{Thm-Estimate-1-for-Support}) we have 
\[
     \supp (\mu^{(b_1, b_2, c)}- \mu^{(b_1, b_2, c)}(\{1\})\delta_1) \subseteq  \mathcal{A}\big[\vartheta_1, \vartheta_2\big],
\]
where $\vartheta_1$ and $\vartheta_2$ are such that $\dsp \cot(\vartheta_2/2) = \inf_{1 \leq n \leq \infty}u_{n+1}$ and $\dsp \cot(\vartheta_1/2) = \sup_{1 \leq n \leq \infty}v_{n+1}$. Clearly,  
\[
   u_{n+1}^{(-)}= - \frac{\sqrt{c^2+q_{n+1}}}{\sqrt{1 - q_{n+1}}}, \quad u_{n+1}^{(+)} = \frac{\sqrt{c^2+q_{n+1}}}{\sqrt{1 - q_{n+1}}}, \quad n \geq 1.
\]
To be able achieve any reasonable results from above, one still needs to find a dominant chain sequence $\hat{d}_{n+1}$ such that $q_{n+1} = d_{n+1}/\hat{d}_{n+1} < 1$, $n \geq 1$, is a scaling sequence for $\{d_{n+1}\} $.

In the particular case when $b_1=b_2=b$, we have $d_{n+1} = (1-b^2)/4$ and by choosing  $\hat{d}_{n+1} = 1/4$ we have 
\[
    u_{n+1}^{(-)}= - \frac{\sqrt{c^2+1-b^2}}{|b|}, \quad u_{n+1}^{(+)} = \frac{\sqrt{c^2+1-b^2}}{|b|}, \quad n \geq 1.
\]
This gives the optimal result, where 
\[
    \vartheta_2 = 2\pi-\vartheta^{+}(b,b, c) \quad \mbox{and} \quad \vartheta_1 = \vartheta^{+}(b,b, c).
\]
When  $b_1 = b_2$ does not hold, then for example when $|b_1| \geq 1/2$, $|b_2| \geq 1/2$ and $b_1b_2 > 0$, we still have $d_{n+1} < 1/4$.  Thus, we obtain for $\vartheta_1$, $\vartheta_2$, 
\[
     \cos(\vartheta_1/2) = \max\Big\{\frac{\sqrt{c^2 + (1+b_1)(1-b_2)}}{\sqrt{c^2+1}}, \frac{\sqrt{c^2 + (1-b_1)(1+b_2)}}{\sqrt{c^2+1}}    \Big\}    
        % \max\Big\{\frac{\sqrt{c^2+(1+b_1)(1-b_2)}}{\sqrt{1 - (1+b_1)(1-b_2)}}, \frac{\sqrt{c^2+(1-b_1)%(1+b_2)}}{\sqrt{1 - (1-b_1)(1+b_2)}}\Big\},
\]
and $\vartheta_2 = 2\pi - \vartheta_1$. 
 
\medskip

\noindent {\bf Example 3.} For $b=\lambda+ i \eta$, $\lambda>-1/2$, $ \eta\in \R$, we consider the nontrivial positive measure $\mu^{(b)}$ on the unit circle given by 
\[
     d \mu^{(b)}(e^{i\theta}) =   e^{-\eta\, \theta}\,  [\sin^{2}(\theta/2)]^{\lambda} d \theta,
\]
where  $\lambda > -1/2$. It was shown in \cite{Ranga-PAMS2010} that the associated Verblunsky coefficients are 
\[
      \alpha_{n-1}^{(b)} = -\frac{(b)_{n}}{(\overline{b}+1)_{n}}, \quad n \geq 1.
\]
As already established in \cite{Costa-Felix-Ranga-JAT2013},  the polynomials $\{R_n\}$ constructed by (\ref{Eq-CDKernel-POPUC}) satisfy the three term recurrence formula (\ref{Eq-TTRR-Rn}), with
\begin{equation*} \label{Eq-Special-Example2}
 \begin{array}l
  \dsp c_n = \frac{\eta}{n+\lambda},  \quad d_{n+1} = d_{n+1}^{(\lambda)} = \frac{1}{4} \frac{n\,(n+2\lambda+1)}{(n+\lambda)(n+\lambda+1)}, \ \ n\geq 1.
 \end{array}
\end{equation*}
Observe that the positive chain sequence $\{d_{n+1}\}_{n=1}^{\infty}$ appearing here coincides with that in \eqref{def:Gegenbauer}, corresponding to the ultra spherical polynomials. 
 
Since $\supp(\mu^{(b)})=\T$, the accumulation set of zeros of $R_n$ (as $n\to\infty$) is the whole  unit circle. Let us  use the results from Section \ref{Sec-BEZeros} %and \ref{Sec-SORPCS} 
in order to find  estimates for the bounds of the extreme zeros of $R_N$ for any fixed $N \geq 2$. 

If $\lambda  \geq 0$ then $d_{n+1} < 1/4$, and a finite  positive chain sequence $\{\hat{d}_{n+1}\}_{n=1}^{N-1}$ dominating $\{d_{n+1}\}_{n=1}^{N-1}$ is 
\begin{equation} \label{Eq-choice1-dn}
    \hat{d}_{n+1} = d^{(N-1)} = \frac{1}{4\cos^{2}(\pi/(N+1))}, \quad n =1, 2, \ldots, N-1,
\end{equation}
(see \cite{IsmailLi-1992}), which gives us the scaling sequence for $\{d_{n+1}\}_{n=1}^{N-1}$:
\begin{equation}\label{example31}
q_{n+1} = \frac{1}{\cos^{2}(\pi/(N+1))} \frac{n(n+2\lambda+1)}{(n+\lambda)(n+\lambda+1)}, \quad n =1, 2, \ldots, N-1.
\end{equation}

The results given in Tables \ref{tbl-1}, \ref{tbl-2} are obtained as an application of Theorem \ref{Thm-Estimates-for-Zero-Bounds-Wn}, when the pair $(\lambda, \eta)$ are respectively  $(1,1)$ and $(10, 0.01)$.  The information within the brackets is to indicate for what value of $n$ the
\[
   \min\{u_n^{(-)}/\sqrt{1+(u_n^{(-)})^2}: 2 \leq n \leq N\}  \  \mbox{and} \  \max\{u_n^{(+)}/\sqrt{1+(u_n^{(+)})^2}: 2 \leq n \leq N\}
\]
 are attained.  For example,  the line corresponding to $N=30$ in Table \ref{tbl-1} informs us $B_{30} =  \max\{u_n^{(+)}/\sqrt{1+(u_n^{(+)})^2}: 2 \leq n \leq 30\} = u_{24}^{(+)}/\sqrt{1+(u_{24}^{(+)})^2}$. 

Clearly, in both tables (\ref{tbl-1} and \ref{tbl-2}),  $2\arccos(B_{N}) < \theta_{N,1}$ and $2\arccos(A_{N}) > \theta_{N,N}$ as expected. 

%TABLE 1
\begin{table}[h] 
\begin{center}
 \begin{tabular}{|r|r|c|r|r|} \hline
\multicolumn{1}{|c|}{$N$} & \multicolumn{1}{c|}{$\approx 2\arccos(B_N)$} &
\multicolumn{1}{c|}{$\approx \theta_{N,1}$} & \multicolumn{1}{c|}{$\approx 2\arccos(A_N)$} & \multicolumn{1}{c|}{$\approx \theta_{N,N}$}  \\[0ex]
\hline
$10$ & $0.4639446$ \ \ \ & $0.4972376$ & $5.4352508$ \ \ \ & $5.1944808$  \\[0ex]
     & $(u_{8}^{(+)})$         &             &        $(u_{10}^{(-)})$ &       \\[0ex]
\hline
$15$ & $0.3198603$ \ \ \ & $0.3499643$ & $5.7029950$ \ \ \ & $5.5126714$  \\[0ex]
     & $(u_{12}^{(+)})$        &             &     $(u_{15}^{(-)})$    &       \\[0ex]
\hline
$30$ & $0.1653904$ \ \ \ & $0.1855341$ & $5.9853660$ \ \ \ & $5.8730792$ \\[0ex]
     & $(u_{24}^{(+)})$        &             &        $(u_{30}^{(-)})$ &      \\[0ex]
\hline
$50$ & $0.1005688$ \ \ \ & $0.1141174$ & $6.1025923$ \ \ \ & $6.0306959$ \\[0ex]
     & $(u_{39}^{(+)})$        &             &        $(u_{50}^{(-)})$ &      \\[0ex]
\hline
\end{tabular} \caption{{\small Information on the bounds for the extreme zeros $e^{i\theta_{N,1}}$ and $e^{i\theta_{N,N}}$ of $R_N^{(b)}$ when $b = 1+i$.  }}\label{tbl-1}
\end{center}
\end{table}
%

%TABLE 2
\begin{table}[h] 
\begin{center}
 \begin{tabular}{|r|r|c|r|r|} \hline
\multicolumn{1}{|c|}{$N$} & \multicolumn{1}{c|}{$\approx 2\arccos(B_N)$} &
\multicolumn{1}{c|}{$\approx \theta_{N,1}$} & \multicolumn{1}{c|}{$\approx 2\arccos(A_N)$} & \multicolumn{1}{c|}{$\approx \theta_{N,N}$}  \\[0ex]
\hline
$10$ & $1.2564079$ \ \ \ & $1.4994620$ & $5.0247247$ \ \ \ & $4.7814017$  \\[0ex]
     & $(u_{10}^{(+)})$         &            &        $(u_{10}^{(-)})$ &       \\[0ex]
\hline
$15$ & $0.9620515$ \ \ \ & $1.1898228$ & $5.3195004$ \ \ \ & $5.0914664$  \\[0ex]
     & $(u_{15}^{(+)})$        &             &        $(u_{15}^{(-)})$ &       \\[0ex]
\hline
$30$ & $0.5731032$ \ \ \ & $0.7410146$ & $5.7090691$ \ \ \ & $5.5409545$ \\[0ex]
     & $(u_{30}^{(+)})$        &             &        $(u_{30}^{(-)})$ &      \\[0ex]
\hline
$50$ & $0.3746598$ \ \ \ & $0.4949570$ & $5.9078531$ \ \ \ & $5.7874076$ \\[0ex]
     & $(u_{50}^{(+)})$        &             &        $(u_{50}^{(-)})$ &      \\[0ex]
\hline
\end{tabular} \caption{{\small Information on the bounds for the extreme zeros $e^{i\theta_{N,1}}$ and $e^{i\theta_{N,N}}$ of $R_N^{(b)}$ when $b = 10+i\,0.01$.  }}\label{tbl-2}
\end{center}
\end{table}

The sequence in \eqref{example31} is increasing in $n$, while for $\eta > 0$ the sequence  $\{c_n\}_{n=1}^{N-1}$ is  decreasing  in $n$. Thus, for $h$ in (\ref{Eq-Inequality-x}) we have that $h(x; c_{n},c_{n+1}) < h(x; c_{n-1},c_{n})$ for $-1 < x < \cos(\vartheta/2)$, where $\cot(\vartheta/2) = c_{n}$. Together with our choice of  $\{q_{n+1}\}_{n=1}^{N-1}$ we see that  $\{u_{n+1}^{(-)}\}_{n=1}^{N-1}$ in Theorem \ref{Thm-Estimates-for-Zero-Bounds-Wn} is  decreasing in $n$. This is confirmed  from the information in brackets in the column associated with $\approx 2\arccos(A_N)$ in Tables \ref{tbl-1} and \ref{tbl-2}.  Therefore, when $\eta > 0$ and $\lambda \geq 0$,  the extreme zero $z_{N,N} = e^{i\theta_{N,N}}$ of $R_{N}$ ($N \geq 2$) is such that 
\[
     \cot(\theta_{N,N}/2) < u_{N}^{(-)} = \frac{2(c_{N-1}c_{N} - q_{N})}{(c_{N-1}+c_{N}) + \sqrt{(c_{N-1}+c_{N})^2 - 4 (1-q_{N})(c_{N-1}c_{N} - q_{N})}}, 
\]
where 
\[
     q_{N} = \frac{1}{\cos^{2}(\pi/(N+1))} \frac{(N-1)(N+2\lambda)}{(N+\lambda-1)(N+\lambda)} \quad \mbox{and} \quad c_{n} = \frac{\eta}{n+\lambda}, \quad n \geq 1.
\]

Now, when $0 > \lambda > -1/2$ we have $d_{n+1} > 1/4$ and the choice (\ref{Eq-choice1-dn}) for $\{\hat{d}_{n+1}\}_{n=1}^{N-1}$ is not viable.  For example, with $\lambda = -1/4$ and $N= 10$ we have $q_2 = d_{2}/d^{(9)} = 1.0521448759.... > 1$. 

Thus, when $0 > \lambda > -1/2$, by considering the results of Lemma \ref{Lemma-Dominant-lambda-FPCS},  we choose $\{\hat{d}_{n+1}\}_{n=1}^{N-1}$ to be 
\begin{equation} \label{Eq-choice2-dn}
     \hat{d}_{n+1}  = \frac{1}{\cos^2(\pi/(2N))}\,d_{n+1}^{(-1/2)}, \quad n =1, 2, \ldots, N-1.
\end{equation}
With this choice, the results given in Tables \ref{tbl-3} are obtained as an application of Theorem \ref{Thm-Estimates-for-Zero-Bounds-Wn}, when  $(\lambda, \eta)$ is  $(-0.25, 1)$.

With the choice of $\{\hat{d}_{n+1}\}_{n=1}^{N-1}$ as in (\ref{Eq-choice2-dn}), again we observe that  $\{q_{n+1}\}_{n=1}^{N-1}$, where
\[
   q_{n+1} = \frac{1}{\cos^{2}(\pi/(2N))} \frac{(n^2-1/4)(n+2\lambda+1)}{n(n+\lambda)(n+\lambda+1)},
\] 
is an increasing sequence as $n$ increase from $1$ to $N-1$. Consequently, we conclude that when $\eta > 0$ and $-1/2 < \lambda < 0$,  the extreme zero $z_{N,N} = e^{i\theta_{N,N}}$ of $R_{N}$ ($N \geq 2$) is such that 
\[
     \cot(\theta_{N,N}/2) < u_{N}^{(-)} = \frac{2(c_{N-1}c_{N} - q_{N})}{(c_{N-1}+c_{N}) + \sqrt{(c_{N-1}+c_{N})^2 - 4 (1-q_{N})(c_{N-1}c_{N} - q_{N})}}, 
\]
where
\[
     q_{N} = \frac{1}{\cos^{2}(\pi/(2N))} \frac{\big((N-1)^2-1/4\big)(N+2\lambda)}{(N-1)(N+\lambda-1)(N+\lambda)} \quad \mbox{and} \quad c_{n} = \frac{\eta}{n+\lambda}, \quad n \geq 1.
\]
%

%TABLE 3
\begin{table}[h] 
\begin{center}
 \begin{tabular}{|r|r|c|r|r|} \hline
\multicolumn{1}{|c|}{$N$} & \multicolumn{1}{c|}{$\approx 2\arccos(B_N)$} &
\multicolumn{1}{c|}{$\approx \theta_{N,1}$} & \multicolumn{1}{c|}{$\approx 2\arccos(A_N)$} & \multicolumn{1}{c|}{$\approx \theta_{N,N}$}  \\[0ex]
\hline
$10$ & $0.1016913$ \ \ \ & $0.1991716$ & $5.6818261$ \ \ \ & $5.2285409$  \\[0ex]
     & $(u_{4}^{(+)})$         &             &  $(u_{10}^{(-)})$ &          \\[0ex]
\hline
$15$ & $0.0635237$ \ \ \ & $0.1358499$ & $5.8881850$ \ \ \ & $5.5600926$  \\[0ex]
     & $(u_{4}^{(+)})$         &             &  $(u_{15}^{(-)})$ &          \\[0ex]
\hline
$30$ & $0.0290353$ \ \ \ & $0.0695512$ & $6.0885290$ \ \ \ & $5.9117387$  \\[0ex]
     & $(u_{7}^{(+)})$         &             &  $(u_{30}^{(-)})$       &    \\[0ex]
\hline
$50$ & $0.0166939$ \ \ \ & $0.0421377$ & $6.1670558$ \ \ \ & $6.0579734$  \\[0ex]
     & $(u_{11}^{(+)})$         &             &  $(u_{50}^{(-)})$       &    \\[0ex]
\hline
 \end{tabular} 
 \caption{{\small Information on the bounds for the extreme zeros 
  $e^{i\theta_{N,1}}$ and $e^{i\theta_{N,N}}$ of $R_N^{(b)}$ when $b = -0.25+i $.  }}\label{tbl-3}
\end{center}
\end{table}

\section*{Acknowledgements}

The first author (AMF) was partially supported by the Spanish Government together with the European Regional Development Fund (ERDF) under grants
MTM2011-28952-C02-01 (from MICINN) and MTM2014-53963-P (from MINECO), by Junta de Andaluc\'{\i}a (the Excellence Grant P11-FQM-7276 and the research group FQM-229), and by Campus 
de Excelencia Internacional del Mar (CEIMAR) of the University of Almer\'{\i}a.

The second author's research was
supported by grants 305073/2014-1 and \linebreak 475502/2013-2 of CNPq and grant 2009/13832-9 of FAPESP.

This research also part of the third author's PhD thesis supported by a grant from CAPES, Brazil. 

Part of this work was carried out during the visit of AMF to the Department of Applied Mathematics of  IBILCE, UNESP. He acknowledges the hospitality of the hosting department, as well as a the financial support of the Special Visiting Researcher Fellowship 401891/2013-5 of the Brazilian Mobility Program ``Science without borders''. 

\newpage 

%%%%%%%%%%%%%%%%%%%%%%%%%%%%%%%%%%%%%%%%%%%%%%%%%%%%%%%%%%%%%%%


\begin{thebibliography}{11}

\bibitem{BracMcCabPerezRanga-MCOM2015} C.F. Bracciali, J.H. McCabe, T.E. Per\'{e}z and A. Sri Ranga, A class of orthogonal functions given by a three term recurrence formula, {\em Math. Comp.}, to appear.

\bibitem{BracRangaSwami-2015} C.F. Bracciali, A. Sri Ranga and A. Swaminathan, Para-orthogonal polynomials on the unit circle satisfying three term recurrence formulas, preprint arXiv:1406.0719. 

\bibitem{CantMoraVela-2002} M.J. Cantero, L. Moral and L. Vel\'{a}zquez,  Measures and para-orthogonal polynomials on the unit circle, {\em East J. Approx.},  {\bf 8} (2002), 447-464.

\bibitem{CantMoraVela-JAT2006} M.J. Cantero, L. Moral and L. Vel\'{a}zquez,  Measures on the unit circle and unitary trunctions of unitary operators, {\em J. Approx. Theory},  {\bf 139} (2006), 430-468.

\bibitem{Castillo-Costa-Ranga-Veronese-JAT2014} K. Castillo, M.S. Costa, A. Sri Ranga and D.O. Veronese, A Favard type theorem for orthogonal polynomials on the unit circle from a three term recurrence formula, {\em J. Approx. Theory}, { 184} (2014), 146--162.

\bibitem{Chihara-Book} T.S. Chihara, {\em An Introduction to Orthogonal Polynomials},
Mathematics and its Applications Series (Gordon and Breach, New
York, 1978).

\bibitem{Costa-Felix-Ranga-JAT2013} M.S. Costa, H.M. Felix and A. Sri Ranga, Orthogonal polynomials on the unit circle and chain sequences, {\em J. Approx. Theory}, { 173} (2013), 14-32. 


\bibitem{DaNjVA-2003} L. Daruis, O. Nj\aa stad and W. Van Assche, Para-orthogonal polynomials in frequency analysis, {\em Rocky Mountain J. Math.} {\bf 33} (2003), 629--645.

\bibitem{DelsarteGenin-1986} P. Delsarte and Y. Genin, The split Levinson algorithm, {\em IEEE Trans. Acoust. Speech Signal Process}, {\bf 34} (1986),
470--478.

\bibitem{DelsarteGenin-1988}  P. Delsarte and Y. Genin,  The tridiagonal approach to Szeg\H{o}'s orthogonal polynomials, Toeplitz linear system, and related interpolation problems, {\em SIAM J. Math. Anal.}, {\bf 19} (1988), 718--735.

\bibitem{Dimitar-Ranga-MN2013} D.K. Dimitrov and A. Sri Ranga, Zeros of a family of hypergeometric para-orthogonal polynomials on the unit circle,  {\em Math. Nachr.}, { 286} (2013), 1778--1791.

\bibitem{DimIsmailRan-2012} D.K. Dimitrov, M.E.H. Ismail and A. Sri Ranga, A class of hypergeometric polynomials with zeros on the unit circle: Extremal and orthogonal properties and quadrature formulas, {\em Appl. Numer. Math.}, { 65} (2013), 41--52.

\bibitem{ENZG1} T. Erd\'elyi, P. Nevai, J. Zhang and J. Geronimo, A simple proof of ``Favard's theorem'' on the unit circle, {\em Atti Sem. Mat. Fis. Univ. Modena}, { 39}  (1991), 551--556.  Also in ``Trends in functional analysis and approximation theory'' (Acquafredda di Maratea, 1989), 41--46, Univ. Modena Reggio Emilia, Modena, 1991. 

\bibitem{Geronimus-AM1946} Ya.L. Geronimus, On the trigonometric moment problem, {\em  Ann. of Math. (2)}, {\bf 4} (1946), 742--761.

\bibitem{Geronimus-AMSTransl-1977}  Ya.L. Geronimus, Orthogonal Polynomials,  English translation of the appendix to the Russian translation of Szeg\H{o}'s book \cite{Szego-Book}, in ``Two Papers on Special Functions'',  Amer. Math. Soc. Transl., Ser. 2, Vol. 108, pp. 37-130, American Mathematical Society, Providence, R.I., 1977. 

\bibitem{Golinskii-2002} L. Golinskii, Quadrature formula and zeros of para-orthogonal polynomials on the unit circle, {\em  Acta Math. Hungar.}, {\bf 96} (2002), 169--186.

\bibitem{GolinNevaiAssche-1995} L. Golinskii, P. Nevai and W. Van Assche, Perturbation of orthogonal polynomials on an arc of the unit circle, { J. Approx. Theory}, { 83} (1995), 392--422. 

\bibitem{GolinNevaiPinterAssche-1999} L. Golinskii, P. Nevai, F. Pint\'{e}r and W. Van Assche, Perturbation of orthogonal polynomials on an arc of the unit circle, II, { J. Approx. Theory}, { 96} (1999), 1--33. 

\bibitem{JoNjTh-1989} W.B. Jones, O. Nj\aa stad and W.J. Thron, Moment theory, orthogonal polynomials, quadrature, and continued fractions associated with the unit circle, {\em Bull. Lond. Math. Soc.}, { 21} (1989), 113--152.

\bibitem{IsmailLi-1992} M.E.H. Ismail and Xin Li, Bounds for the extreme zeros of orthogonal polynomials,  {\em Proc. Amer. Math. Soc}, { 115} (1992), 131--140.

\bibitem{JoNjTh-BLMS1989} W.B. Jones, O. Njåstad and W.J. Thron, Moment theory, orthogonal polynomials, quadrature, and continued fractions associated with the unit circle, {\em Bull. Lond. Math. Soc.}, { 21} (1989), 113--152.

\bibitem{MFMLS}
A.~Mart\'{\i}nez-Finkelshtein, K.T.-R.~McLaughlin, and E.B.~Saff, Szeg\H{o} orthogonal polynomials with respect to an analytic weight: canonical representation and strong asymptotics. {\em Constructive Approximation} 24 (3), (2006) , 319--363.

\bibitem{MFMLS2}
A.~Mart\'{\i}nez-Finkelshtein, K.T.-R.~McLaughlin, and E.B.~Saff, Asymptotics of orthogonal polynomials with respect to an analytic weight with algebraic singularities on the circle. {\em IMRN} (2006), Art. ID 91426, 43 pp.

%\bibitem{Schinzel-2005} A. Schinzel, Self-inversive polynomials with all zeros on the unit circle, {\em Ramanujan J.}, 9 (2005), 19-23.

\bibitem{Simon-Book-p1} B. Simon, \emph{Orthogonal Polynomials on the Unit Circle. Part 1. Classical Theory},  Amer. Math. Soc. Colloq. Publ., vol. { 54}, part 1, Amer. Math. Soc., Providence, RI,  2005.


\bibitem{Simon-Book-p2} B. Simon, \emph{Orthogonal Polynomials on the Unit Circle. Part 2. Spectral Theory},  Amer. Math. Soc. Colloq. Publ., vol. { 54}, part 1, Amer. Math. Soc., Providence, RI,  2005.

\bibitem{Simon-2011} B. Simon, \emph{``Szeg\H{o}'s Theorem and Its Descendants: Spectral Theory for L$^2$ perturbations of Orthogonal Polynomials''}, Princeton Univ. Press, Princeton, 2011.

\bibitem{Sinclair-Vaaler-2008} C.D. Sinclair and J.D. Vaaler, Self-inversive polynomials with all zeros on the unit circle, in Number Theory and Polynomials (J. McKee and C. Smyth, eds.), London Mathematical Society Lecture Notes Series, vol. 352, Cambridge Univ. Press, 2008.


\bibitem{Ranga-PAMS2010} A. Sri Ranga, Szeg\H{o} polynomials from hypergeometric functions, {\em Proc. Amer. Math. Soc.}, {\bf 138} (2010), 4259--4270.


\bibitem{Szego-Book} G. Szeg\"{o}, {\em Orthogonal Polynomials}, 4th ed., Amer. Math. Soc. Colloq. Publ., vol. {\bf 23}, Amer. Math. Soc., Providence, RI, 1975.


\bibitem{Wall-Book} H.S. Wall, \emph{Analytic Theory of Continued Fractions},  D. van Nostrand, 1948.


\bibitem{Wong-2007} M.L. Wong, First and second kind paraorthogonal polynomials and their zeros, {\em  J. Approx. Theory}, {\bf 146} (2007), 282--293.

\bibitem{Zhedanov-JAT98} A. Zhedanov, On some classes of polynomials orthogonal on arcs of the unit circle connected with symmettric orthogonal polynomials on an interval, { J. Approx. Theory}, { 94} (1998), 73--106. 

\end{thebibliography}
\end{document}